\documentclass[12pt] {amsart}
\usepackage{amsthm, mathtools, amsmath,amscd,graphicx,amssymb,
amstext, epsfig,verbatim, multicol,mathrsfs,mathbbol}
\usepackage[T2A]{fontenc}
\usepackage[utf8x]{inputenc}
\newtheorem{thm}{Theorem}
                                             
\newtheorem{lem}{Lemma}

\newtheorem{cor}{Corollary}[section]
{\theoremstyle{remark}
\newtheorem {Rem}{Remark}}
{\theoremstyle{definition}
\newtheorem{example}{Example}}
                         \newtheorem{defn}[thm]{Definition}
\def\div{\mathrm{div}}

\def\fchar{\mathrm{char}}

\def\supp{\mathrm{sup}}

                                                   \newcommand{\CC}{\mathcal C}

\newcommand{\al}{\alpha}

\newcommand{\eps}{\varepsilon}

                                             \newcommand{\alb}{\textit{alb}}
\newcommand{\beq}{\begin{equation}}
\newcommand{\eeq}{\end{equation}}

                                                 \newcommand{\W}{\mathfrak W}

\newcommand{\QQ}{\mathbb Q}

\newcommand{\Eta}{\mathrm{H}}

\title[Torsion points on hyperelliptic curves]{Torsion Points of order $\boldsymbol{2g+1}$ on odd degree hyperelliptic curves of genus $\boldsymbol{g}$}
\author {Boris M. Bekker}
\address{St. Petersburg State University, Department of Mathematics and Mechanics,   Universitates prospect, 28, Peterhof, St. Petersburg, 198504, Russia}
\email{ bekker.boris@gmail.com}
\author {Yuri G. Zarhin}
\address{Pennsylvania State University, Department of Mathematics, University Park, PA 16802, USA}
\email{zarhin@math.psu.edu}
\thanks{The second named author (Y.Z.) was partially supported by Simons Foundation Collaboration grant   \# 585711.
This project was started in May-July 2018 when Y.Z. visited
 the Max Planck Institut f\"ur Mathematik (Bonn, Germany), whose hospitality and support are gratefully acknowledged.
Both authors are grateful to the organizers of the workshop
``Arithmetic of curves"  (August 2018) at Baskerville Hall (Wales) for excellent working conditions that enabled them
to complete the project.}

\begin{document}
\begin{abstract}
Let $K$ be an algebraically closed field of characteristic different from $2$, $g$ a positive integer, $f(x)\in K[x]$ a degree $2g+1$ monic polynomial without multiple roots, $\mathcal{C}_f: y^2=f(x)$ the corresponding genus $g$ hyperelliptic curve over $K$, and $J$ the Jacobian of $\mathcal{C}_f$. We identify $\mathcal{C}_f$ with the image of its canonical embedding into $J$ (the infinite point of $\mathcal{C}_f$ goes to the zero of the group law on $J$). It is known \cite{Zarhin} that if $g\ge 2$, then $\mathcal{C}_f(K)$  contains no  points  of orders  lying between $3$ and $2g$.

In this paper we study torsion points of order $2g+1$ on $\mathcal{C}_f(K)$. Despite  the striking difference between the cases of $g=1$ and $g\ge 2$,  some of our results may be viewed as a generalization of well-known results about points of order $3$ on elliptic curves. E.g.,   if  $p=2g+1$ is a prime that coincides with $\fchar(K)$, then every  odd degree genus $g$ hyperelliptic curve contains at most two points of order $p$. If  $g$ is {\sl odd} and  $f(x)$ has {\sl real} coefficients,  then there are at most two real points of order $2g+1$ on $\mathcal{C}_f$. If  $f(x)$ has {\sl rational} coefficients and $g\le 51$, then
 there are at most two rational points of order $2g+1$ on $\mathcal{C}_f$.  (However,  there   exist  odd degree  genus $52$ hyperelliptic curves over $\mathbb{Q}$ that have at least four rational points of order 105.)
\end{abstract}



\maketitle

\section{Introduction}
\label{introd}
Let $K$ be an algebraically closed field and $K_0$ a subfield of $K$. Let $\CC$   be a smooth irreducible projective curve of {\sl positive} genus $g$ over $K$. We say that $\CC$ is defined over $K_0$ if there exists a smooth  projective curve
   $\CC_0$  of the same genus $g$  over $K_0$ such that $\CC=\CC_0\times_{K_0}K$.  Let $\CC$ be defined over $K_0$, $O$ be a $K_0$-point on $\CC$, and $J$ the Jacobian of $C$, which is a $g$-dimensional abelian variety over $K_0$.
There is a canonical $K_0$-regular embedding $\alb:\CC \to J$ that sends $O$ to the zero of the group law on $J$   and every point $P \in \CC(K)$ to the linear equivalence class of the divisor $(P)-(O)$. We identify $\CC$ with its image in $J$.

A celebrated theorem of Raynaud (Manin--Mumford conjecture) asserts that if $g>1$ and $\mathrm{char}(K)=0$, then the set of torsion points in $\CC(K)$ is finite \cite{Raynaud}. (This assertion does not hold in prime characteristic: e.g., if $K$ is an algebraic closure of a finite field, then $\CC(K)$ is an infinite set that consists of torsion points.) The importance of  determining explicitly the finite set occuring   in Raynaud's theorem
for specific curves was stressed  by R.Coleman,  K.Ribet, and other authors (for more details, see \cite{Ribet}, \cite{Tsermias}).

In particular, it is  natural  to ask what are the ``small'' orders of torsion points in $\CC(K_0)\subset J(K_0)$ for specific $K_0$.
More precisely:\begin{itemize}\item
    does  there exist (for   given $g>1$, $K_0$, and a positive integer $n$) a genus $g$ curve $\CC$ over $K_0$ such that $\CC(K_0)$ contains a torsion point of order $n?$

    \item if such a curve exists, then how many   points of order $n$ it may contain?
  \end{itemize}

     Of course, $O\in \CC(K)\subset J(K)$ is the only point of order $1$, so we may assume that $n>1$.
In what follows we assume that $\fchar(K)\ne 2$, $\CC$ is  a genus $g$ {\sl hyperelliptic} curve, and $O$ is  one of the {\sl Weierstrass points} of $\CC$.
It is well known that the torsion points of order $2$ in $\CC(K)$ are  the remaining (different from $O$) $(2g+1)$ Weierstrass points on $\CC(K)$.

  It was proven by J. Boxall and D. Grant in \cite{Box1} that
for $g=2$ there are no points of order $3$ or $4.$

  The second named author proved in   \cite[Theorem 2.8]{Zarhin} that   $\mathcal \CC(K)$  does {\sl not} contain a point of order $n$ if $g \ge 2$ and
 $3\leq n\leq 2g$.

Thus  the first nontrivial case is $n=2g+1$,  which is the subject of the present paper.

	  In the case of $g=2$ such  a study  was done by
 J. Boxall, D. Grant, and F. Lepr\'evost
    in \cite{Box2}, where  a  classification (parameterization) of the genus 2 curves  (up to an isomorphism)  with torsion points of order 5 over algebraically closed fields  was given.  In particular, it was proven in \cite{Box2} that if $\fchar(K)=5$,
     then $\mathcal C(K)$ contains at most 2 points of order 5.  The latter assertion may be viewed as a genus $2$  analog  of the following well-known fact: an elliptic curve in characteristic 3 has at most 2 points of order 3.

Here are our results about genus $g$ hyperelliptic curves   $\CC $  with torsion points of order $n=2g+1$.

\begin{itemize}\item[1.]  For every $K_0$ and every $g$ there exists a curve
$\CC$ such that $\CC(K_0)$ contains at least two points of order  $n=2g+1$
(Examples \ref{exNotDivide} and \ref{exDivide}, Remark \ref{exist2gp1}). Actually, we construct a versal family of such curves that is parameterized by an affine rational $K_0$-variety (Theorems 2 and 1).

\item[2.]  If $K_0$ is the field $\mathbb Q$ of rational numbers and $g\le 100$,   then  a curve $\CC$  having at least four points of order  $n=2g+1$ in
$\CC(\QQ)$ exists if and only if $g=52$ or $g=82$
(Corollary \ref{overQ} and Example  \ref{EXoverQ}).

\item[3.] If $K_0$ is the field $\mathbb R$  of real numbers,  then    a curve $\CC$  having at least four points of order  $n=2g+1$ in  $\CC({\mathbb R})$
 exists if and only if $g$ is {\sl even}
(Corollary \ref{realNO}) and Example \ref{realYES}).
\item[4.]  If $p=\fchar(K)>0$ and   $2g+1$ is a power of $p$ (e.g., $2g+1=p$),  then  $\CC(K)$ contains at  most  two points of order  $n=2g+1$  for every
$\CC$ (Theorem \ref{char2gPlus1}).
\item[5.]  If  $2g+1$ is {\sl not} a power of $\fchar(K)$ (e.g., $\fchar(K)=0$ or $\fchar(K)>2g+1$),  then there exists a versal family  of curves having at least four   points of order  $n=2g+1$ in
$\CC(K).$ This family is parameterized by a finite (nonempty) disjoint union of affine rational $K$-curves (Theorems \ref{theorem20} and \ref{theorem22}).
 \end{itemize}

While there are   striking differences between the cases of $g=1$ (elliptic curves) and $g\ge 2$,  our assertions 1, 3, and 4 may be viewed as  generalizations of well-known results about points of order $3$ on elliptic curves.

 The paper is organized as follows. In Section \ref{prel} we remind the reader of some basic results about odd degree hyperelliptic curves.
 It also contains  auxiliary assertions from \cite{Zarhin} that will be used later. In Section \ref{twoPoints} we describe odd degree genus $g$ hyperelliptic curves with one pair of torsion points of order $2g+1$.  It turns out that such curves and points exist over arbitrary fields for all $g$ (Examples \ref{exNotDivide} and \ref{exDivide}).  We give a characterization of hyperelliptic genus $g$ curves with two  pairs of torsion points of order $2g+1$ in terms of certain factorizations of the polynomial $(x-a_2)^{2g+1}-(x-a_1)^{2g+1}$ where $a_1$ and $a_2$ are abscissas of the torsion points. Each such factorization gives rise to a   one-dimensional family of hyperelliptic genus $g$ curves with two  pairs of torsion points of order $2g+1$,    and we study them in
Section~ \ref{famCurves}. In Section~ \ref{nonex} we discuss the rationality questions and prove the results over $\mathbb{R}$ and $\mathbb{Q}$ mentioned above. We also introduce and discuss the notion of {\sl hyperelliptic} numbers, which may be of independent interest. In Section~ \ref{sec3} we concentrate on the case of algebraically closed field. We study  odd degree genus $g$ hypelliptic curves that have two torsion points $P,Q$ of order $2g+1$ with $P\ne Q, P\ne \iota(Q)$ and  provide a parameterization of their isomorphism classes by a disjoint union of
 finitely many  affine rational curves. In Section \ref{WeilComp}  we compute the value of the Weil pairing between certain torsion points of order $2g+1$ on $\mathcal{C}_f$.

\section{Preliminaries}
\label{prel}

Let  $K$ be an algebraically closed field with $\fchar(K)\ne 2$.
Let  $\mathcal C$  be a hyperelliptic curve of genus $g\geq 1$  over $K$.
 Let $K(\mathcal C)$ be the field of rational functions on $\mathcal C$ and  $J$ the Jacobian of $\mathcal C$.
  Let $O\in \mathcal{C}(K)$ be  a {\sl Weierstrass}
 point on $\mathcal C$.  The  pair $(\mathcal C,O)$ is called a {\sl pointed} hyperelliptic curve \cite{Lock}.  (If $g=1$, then every $K$-point of $\mathcal C$ is a Weierstrass one. If $g>1$, then there are exactly $2g+2$ Weierstrass $K$-points on $\mathcal C$ .)
 By the definition of a Weierstrass point \cite{Lock},
 there exists a rational function $x\in K(\mathcal C)$  that is regular outside $O$ and has a double pole at $O$. (Any other rational function on $\mathcal C$
 that enjoys these properties is of the form $\alpha x+\beta$ with $\alpha\in K^{*},\beta \in K$ \cite{Lock}.) The regular map $\pi:\mathcal C \to \mathbb{P}^1$
 to the projective line $\mathbb{P}^1$ defined by $x$ is a {\sl double cover} that sends $O$ to the infinite point of  $\mathbb{P}^1$. The $K$-biregular {\sl involution}
 $$\iota=\iota_{\mathcal C}:\mathcal C \to \mathcal C$$
 attached to $\pi$
  is the so-called {\sl hyperelliptic involution} of the hyperelliptic curve $\mathcal C$, which does {\sl not} depend on a choice of $x$;   it even does not depend on a choice of  $O$  if $g>1$.  The set of fixed points of $\iota$ (i.e., the set of  branch points of $\pi$) is a certain $(2g+2)$-element set of  Weierstrass points in $\mathcal C(K)$, including $O$. (If $g>1$, then this set coincides with the set of all Weierstrass points on   $\mathcal C$.) The $K$-vector subspace $\mathcal{L}((2g+1)(O))\subset K(\mathcal C)$ of functions that are regular outside $O$ and have a pole of order  at most  $2g+1$ at $O$ has dimension
 $g+2$; in addition, it is $\iota$-stable and contains $g+1$ linearly independent $\iota$-invariant functions $1,x, \dots, x^g$ that have a pole of order at most $2g$ at $O$ \cite{Lock}. This implies that
 there exists a rational function $y\in K(\mathcal C)$  that is $\iota$-anti-invariant, regular outside $O$, and has a pole of order $2g+1$ at $O$; such a  $y$ is unique up to multiplication by a nonzero element of $K$.
 In addition, there exists a degree $2g+1$
 polynomial $f(x)\in K[x]$ without multiple roots such that $y^2=f(x)$ in $K(\mathcal C)$ \cite{Lock}. Multiplying $x$ and $y$ by suitable nonzero elements of $K$, we may and will assume that $f(x)$ is monic.
 The functions $(x,y)$ define a biregular $K$-isomorphism between $\mathcal C$ and the (smooth) normalization $\mathcal C_f$  of the projective closure of the smooth plane affine curve $y^2=f(x)$ under which $O$ goes to the unique {\sl infinite} point of $\mathcal C_f$ \cite{Lock}, which we denote by $\infty$; in addition, $\iota_{\mathcal C}$ becomes the involution
 $$\mathcal C_f \to \mathcal C_f,  \ (x,y)\mapsto (x,-y).$$
The fixed points of $\iota$ are
  $\infty$ and all the points $\W_i=(w_i,0)$, where $w_i\in K$ ($1\leq i\leq 2g+1$) are the roots of   $f(x)$.

The action of $\iota$ on $\mathcal C(K)$ extends by linearity to the action on divisors of $\mathcal C$.
  Notice that for any nonzero rational function
   $F$ on $\mathcal C$ we have  $\div(\iota^{\ast}(F))=\iota(\div F),$ where $\div(F)$ is the divisor of
$F$ and $\iota^{\ast}$ is the induced action of $\iota$ on the field of rational functions on $\mathcal C$.  Thus we obtain the induced action of $\iota$ on the linear equivalence classes of divisors on $\mathcal C$.
 If $P\in \mathcal C(K)$, then we write $(P)$ for the corresponding degree 1 effective divisor with support in $P$. If $P=(a,b)$, then $\div(x-a)=(P)+(\iota(P))-2(\infty)$. This explains why after the identification of $\mathcal C$ with its image in $J$ the involution $\iota$ becomes multiplication by $-1$ and the points of order 2 in  $\mathcal C(K)$  are all (except $\infty$)
 $(2g+1)$
  branch points of $\pi$
  of  $\mathcal C$.
   Notice that if
   $\mathcal C(K)$ contains a torsion point $P$ of order $n>2$, then it contains the torsion point $\iota(P)\ne P$
  of the same order, which implies that
  the number of  points of order $n$ in $\mathcal C(K)$ is even.

  If $K_0$ is a subfield of $K$, $\CC$ is defined over $K_0$, and $O\in\CC(K_0)$, then $(\mathcal C,O)$ is called a {\sl pointed  hyperelliptic curve  \sl over}  $K_0$ \cite{Lock}. Let $(\mathcal C,O)$ be a pointed  hyperelliptic curve over $K_0$ and $K_0(\CC)$ be the field of $K_0$-rational functions on $\CC$. Then the coordinate functions $x$ and $y$
 can be chosen in $K_0(\CC)$, the double cover $\pi:\mathcal C \to \mathbb{P}^1$ defined by $x$ is $K_0$-regular, and the involution $\iota_{\mathcal C}:\mathcal C \to \mathcal C$ is $K_0$-biregular. In addition, the monic polynomial $f(x)$ such that $y^2=f(x)$ can be chosen in $K_0[x]$.  The functions $(x,y)$ define a biregular $K_0$-isomorphism between $\mathcal C$ and the (smooth) normalization $\mathcal C_f$  of the projective closure of the smooth plane affine curve $y^2=f(x)$ over $K_0$ under which $O$ goes to  $\infty$ \cite{Lock}.

  We  say that two pointed hyperelliptic curves $(\mathcal C_1, O_1)$ and $(\mathcal C_2, O_2)$ are isomorphic if there exists a $K_0$-regular isomorphism $\phi: \mathcal C_1\to \mathcal C_2$ such that $\phi(O_1)=O_2$.
Each isomorphism class of pointed hyperelliptic  curves over $K_0$ contains a pointed hyperelliptic curve  of the form $(\mathcal C_f,\infty)$. In what follows, we may assume without loss of generality that $\mathcal C=\mathcal C_f$ for a suitable $f(x)\in K_0[x]$ and $O=\infty$.
  \begin{Rem}
  \label{isoPoint} Let $(\CC,\infty)$ and $(\CC_1,\infty_1)$ be pointed genus $g$ hyperelliptic curves over $K_0$ defined respectively by
$y^2=f(x)$ and  $ y_1^2=f_1(x_1)$, where $f(x),f_1(x)\in K_0[x]$ are monic degree $2g+1$ polynomials without multiple roots.
 Let $\phi: (\mathcal C,\infty) \cong (\mathcal C_1,\infty_1)$ be an isomorphism of pointed hyperelliptic curves over $K_0$, i.e.,
  a $K_0$-biregular isomorphism $\CC \to \CC_1$ of $K_0$-curves that sends $\infty$ to $\infty_1$.  Then there exist $\lambda \in K_0^{*}$ and $r\in K_0$
  such that
  $$\phi^{*}(x_1)=\lambda^2 x+r \in K_0(\mathcal C), \ \phi^{*}(y_1)=\lambda^{2g+1} y \in K_0(\mathcal C)$$
  (see \cite[Prop. 1.2 and Remark on p. 730]{Lock}).
  This implies that in $K_0(C)$
  $$(\lambda^{2g+1} y)^2=f_1(\lambda^2 x+r)$$
  and therefore
  $$y^2=\frac{f_1(\lambda^2 x+r)}{\lambda^{2(2g+1)}}.$$
  Consequently,
  $$f(x)=\frac{f_1(\lambda^2 x+r)}{\lambda^{2(2g+1)}}$$
  and therefore
  $$f_1(x)=\lambda^{2(2g+1)} \cdot f\left(\frac{x-r}{\lambda^2}\right).$$
  Assume additionally that
  $f(0) \ne 0, f_1(0)\ne 0$, and $\phi$ sends a point $P=\left(0,\sqrt{f(0)}\right)\in \mathcal{C}(K)\setminus \{\infty\}$
  with abscissa $0$ to a point $P_1\in \mathcal{C}_1(K)\setminus  \{\infty\}$
  with abscissa $0$. Then $r=0$ and
  \beq\label{changeL}
   \phi^{*}(x_1)=\lambda^2 x, \phi^{*}(y_1)=\lambda^{2g+1} y, \ f_1(x)=\lambda^{2(2g+1)} \cdot f\left(\frac{x}{\lambda^2}\right).
   \eeq
   Let us assume also that there are  nonzero  $a_,b\in K_0$  such that
    $$f(a) \ne 0, f_1(b)\ne 0$$ and $\phi$ sends a point $Q=(a,\sqrt{f(a)})\in \mathcal{C}(K)\setminus \{\infty\}$
  with abscissa $a $ to a point $Q_1\in \mathcal{C}_1(K)\setminus  \{\infty\}$
  with abscissa $b$.
   Then $b=x_1(Q)=\lambda^2 x(P)=\lambda^2 a$, i.e.,
   \beq\label{aLambdab}
   \lambda^2=\frac{b}{a},\  \lambda=\sqrt{\frac{b}{a}},
  \eeq
  Since $\lambda \in K_0$, we conclude that $b/a$ is a square in $K_0$. In addition
   \beq\label{fLambdaf1}
   f_1(x)=\lambda^{2(2g+1)} \cdot f\left(\frac{x}{\lambda^2}\right)=\left(\frac{b}{a}\right)^{2g+1} f\left(\frac{x}{b/a}\right).
  \eeq
  In particular, if $a=b$, then $b/a=1$ and therefore $f(x)=f_1(x)$,
  i.e., $\mathcal C=\mathcal C_1$
   and either
  $$\lambda=1, \phi^{*}(x_1)=x, \phi^{*}(y_1)=y_1$$
   and $\phi$ is the identity map
  or
  $$\lambda=-1, \phi^{*}(x_1)=x, \phi^{*}(y_1)=-y_1$$
  and $\phi=\iota$.
  \end{Rem}

We will need the following assertion that was proven in \cite{Zarhin}.

\begin{lem}\label{l1}
Let $D$ be an effective positive degree $m$ divisor on $\mathcal C$  such that $m \le 2g+1$ and $\supp(D)$ does not contain $\infty$. Assume that the divisor $D-m(\infty)$ is principal.

\begin{enumerate}
\item[(1)]
Suppose that $m$ is odd.
 Then:

 \begin{itemize}
 \item[(i)]
  $m=2g+1$ and there exists exactly one polynomial $v(x)\in K[x]$ such that  the divisor of $y-v(x)$ coincides with $D-(2g+1)(\infty)$. In addition, $\deg(v)\le g$.
    \item[(ii)]
    If $\W_i$ lies in  $\supp(D)$, then it appears in $D$ with multiplicity 1.
    \item[(iii)]
    If $b$ is a nonzero element of $K$ and  $P=(a,b) \in \mathcal C(K)$ lies in  $\supp(D)$, then $\iota(P)=(a,-b)$ does not lie in  $\supp(D)$.
  \end{itemize}
  \item[(2)]
  Suppose that $m=2d$ is even. Then there exists exactly one monic degree $d$ polynomial $u(x)\in K[x]$ such that  the divisor of $u(x)$ coincides with $D-m(\infty)$.  In particular, every point $Q \in \mathcal C(K)$ appears in $D-m(\infty)$ with the same multiplicity as $\iota(Q)$.
  \end{enumerate}
\end{lem}

We finish this section  by the following elementary useful statement.

 \begin{lem}\label{l2}
 Let $K_0$ be a field, let $a$ be a nonzero element of $K$, and let $w(x) \in K_0[x]$ be a degree $g$ polynomial with nonzero constant term. Then there exists a unique
 degree $g$ polynomial $\tilde{w}(x) \in K_0[x]$ with nonzero constant term such that in the field $K_0(x)$ of rational functions
 $$\tilde{w}(a/x)=\frac{w(x)}{x^g}.$$
 \end{lem}

 \begin{proof}
 We have
 \beq
 \label{xax}
 w(x)=\sum_{i=0}^g b_i x^i, \  a_i \in K_0,\  b_0 \ne 0,\ b_g\ne 0.
 \eeq
 Then
 $$\frac{w(x)}{x^g}=\sum_{i=0}^g b_i x^{i-g}=\sum_{i=0}^g \frac{b_i}{a^{g-i}} (a/x)^{g-i}.$$
 Let us put
 $$\tilde{w}(x)=\sum_{i=0}^g \frac{b_i}{a^{g-i}}  x^{g-i}\in K_0[x].$$
 Clearly, $\deg(\tilde{w}) \le g$. The coefficient of $\tilde{w}$ at $x^g$ is $b_0/a^g \ne 0$, and therefore $\deg(\tilde{w}) = g$.
 The constant term of $\tilde{w}$ is $b_g \ne 0$. It follows from \eqref{xax} that
 $$\tilde{w}(a/x)=\frac{w(x)}{x^g}.$$
 The uniqueness of $\tilde{w}$ is obvious.
\end{proof}

\section{Torsion  points of order $2g+1$}
\label{twoPoints}

The next assertion describes all odd degree hyperelliptic curves of genus $g$ that admit a torsion point of order
 $2g+1$.
\begin{thm}\label{ord 2g+1}
Let $g \ge 1$ be an integer and $f(x)\in K[x]$  a monic degree  $2g+1$ polynomial without multiple roots.  Then the odd degree hyperelliptic curve
  $y^2=f(x)$ has a point $P$ of order $2g+1$ if and only if there exist $a\in K$ and a polynomial  $v(x)\in K[x]$  such that
   $$\deg(v)\le g, \ v(a)\neq0, \ f(x)=(x-a)^{2g+1}+v^2(x).$$
If this is the case, then the point $P=(a,v(a)) \in \mathcal C(K)$ has order $2g+1$.
\end{thm}
\begin{proof}
Suppose that $P=(a,c)$  is a $K$-point on  $\mathcal C$ having order  $2g+1$ in  $J(K)$. Then the divisor
 $(2g+1)(P)-(2g+1)(\infty)$ is principal.  By Lemma \ref{l1}, there exists precisely one polynomial  $v(x)$  with $\deg(v)\le g$ such that
$$\div(y-v(x))=(2g+1)(P)-(2g+1)(\infty).$$
Thus the zero divisor of
 $y-v(x)$ coincides with $(2g+1)(P)$. In particular,  $c=v(a)$. Notice that the point $\iota(P)=(a,-c)$ also has order $2g+1$. The zero divisor of $y+v(x)$ equals $(2g+1)(\iota(P))$. Since $P\neq \iota(P)$, the zero divisor of $$y^2-v^2(x)=f(x)-v^2(x)$$ equals $(2g+1)(P)+(2g+1)(\iota(P))$ while its polar divisor is $2(2g+1)(\infty)$.
This means that the monic degree $ 2g+1 $ polynomial
$f(x)-v^2(x)$ equals $(x-a)^{2g+1}$, which implies that $f(x)=(x-a)^{2g+1}+v^2(x)$.

Conversely, let us consider the  hyperelliptic curve
 $y^2=(x-a)^{2g+1}+v^2(x)$, where $v(x)\in K[x]$ is a polynomial with $\deg(v) \le g$ and
  $v(a)\neq0$.  Let us put   $c=v(a)$ and prove that $P=(a,c)\in \mathcal C(K)$ has order $2g+1$. It follows from
 $y^2-v^2(x)=(x-a)^{2g+1}$ that all zeros of  $y-v(x)$  have abscissa $a$. Clearly,
 $P=(a,c)$ is a zero of  $y-v(x)$, but  $\iota(P)=(a,-c)$ is not one a zero of  $y-v(x)$, because $y-v(x)$ takes the value $-c-v(a)=-2v(c)\ne 0$
 at $\iota(P)$.  This implies that   $y-v(x)$ has exactly one zero, namely $P$. Obviously, $y-v(x)$ has exactly one pole, namely $\infty$,
 and its multiplicity is $2g+1$. It follows that
 $$\div(y-v(x))=(2g+1)(P)-(2g+1)(\infty)=(2g+1)((P)-(\infty)).$$
 This implies that $P$ has finite order $m$ in $J(K)$ and $m$ divides $2g+1$. Clearly, $m$ is neither 1 nor 2.  If $g=1$, then $2g+1=3$ is a prime  divisible by $m$. This implies that $m=3=2g+1$, i.e., $P$ is a torsion point of order $2g+1$.    Now assume that $g>1$.
 By a result of \cite{Zarhin}, $m$ cannot lie between $3$ and $2g$. This implies again that $m=2g+1$, i.e., $P$ is a torsion point of order $2g+1$.

\end{proof}

\begin{example}
\label{exNotDivide}
Suppose that $\fchar(K)$ does {\sl not} divide $2g+1$.  
Then the polynomial $x^{2g+1}+1$ has {\sl no} multiple roots and the odd degree genus $g$ hyperelliptic curve
$$y^2=x^{2g+1}+1$$
contains a torsion point $(0, 1)$ of order $2g+1$ \cite[example 2.10]{Zarhin}.
\end{example}

\begin{example}
\label{exDivide}
Suppose that $\fchar(K)$ divides  $2g+1$.  Choose a {\sl nonzero} $b \in K$.
Then the polynomial $f(x)=x^{2g+1}+(bx+1)^2$ has no multiple roots. Indeed,
$f^{\prime}(x)=2b(bx+1)$. So, if $x_0$ is a root of $f^{\prime}(x)$, then $b x_0+1=0$,
which implies that $x_0 \ne 0$ and
$$f(x_0)=x_0^{2g+1}+(bx_0+1)^2=x_0^{2g+1} \ne 0.$$
This proves that $f(x)$ has {\sl no} multiple roots. Applying Theorem \ref{ord 2g+1} to $a=0$ and $v(x)=bx+1$,
we conclude that the odd degree genus $g$ hyperelliptic curve
$$y^2=x^{2g+1}+(bx+1)^2$$ has  a torsion point $P=(0,1)$ of order $2g+1$.   If we take $b=1$, then we obtain that the odd degree genus $g$ hyperelliptic curve $y^2=x^{2g+1}+(x+1)^2$ has two torsion points $(0,\pm 1)$ of order $2g+1$.
\end{example}

\begin{Rem}
\label{isomOneP}
Let $v(x), w(x)\in K[x]$ be polynomials  whose degrees do not exceed $g$ with
$$v(0)\ne 0, \ w(0)\ne 0$$
 and such that both degree $2g+1$ polynomials
$$f(x)=x^{2g+1}+v^2(x), \ f_1(x)=x^{2g+1}+w^2(x)$$
have no multiple roots. Let us consider the odd degree genus $g$ hyperelliptic curves
$$\mathcal{C}: y^2=x^{2g+1}+v^2(x) \ \text{ and } \ \mathcal{C}_1: y_1^2=x_1^{2g+1}+w^2(x_1)$$
over $K$. By Theorem \ref{ord 2g+1}, $P=(0,v(0))$ is a  torsion point of order $2g+1$ in $\mathcal{C}(K)$
and $P_1=(0,w(0))$ is a torsion point of order $2g+1$ in $\mathcal{C}_1(K)$.
It follows from arguments of Remark
\ref{isoPoint} that if there is a $K$-biregular isomorphism of pointed curves $\phi:\mathcal C\cong \mathcal C_1$
that sends $P$ to $P_1$, then there exists $\lambda\in K^{*}$
such that
$$\phi^{*}x_1=\lambda^2 x, \ \phi^*{y_1}=\lambda^{2g+1} y,$$
$$x^{2g+1}+w^2(x)=f_1(x)=\lambda^{2(2g+1)} \cdot f\left(\frac{x}{\lambda^2}\right)=x^{2g+1}+\lambda^{2(2g+1)}\left(v\left(\frac{x}{\lambda^2}\right)\right)^2.$$
This implies that
$$w(x)=\pm \lambda^{(2g+1)} v\left(\frac{x}{\lambda^2}\right).$$
\end{Rem}

\begin{thm}\label{ordK02g+1}
Let $K_0$ be a subfield of $K$.
Let $g \ge 1$ be an integer and $$f(x)\in K_0[x]\subset K[x]$$ be  a monic degree  $2g+1$ polynomial without multiple roots.

Suppose that  the  hyperelliptic curve
  $C_f: y^2=f(x)$ has a $K_0$-point $P=(a,c)$ of order $2g+1$.
 Then there exists  precisely one polynomial  $v(x)\in K_0[x]$  such that
   $$\deg(v)\le g, \ v(a)=c\neq0, \ f(x)=(x-a)^{2g+1}+v^2(x).$$
\end{thm}

\begin{proof}
It follows from Theorem \ref{ord 2g+1} and its proof that
there exists  a polynomial  $v(x)\in K[x]$  such that
   $$\deg(v)\le g, \ v(a)=c\neq0, \ f(x)=(x-a)^{2g+1}+v^2(x).$$
   Since $f(x) \in K_0[x]$,  we get $v^2(x)\in K_0[x]$.
   This implies that the polynomial $w(x)=v(x)/c$ satisfies
    $$w(a)=1, \ w^2(x)\in K_0[x].$$
  If we put $\tilde{w}(x)=w(x+a)\in K[x]$, then
    $$\tilde{w}(0)=1,   \tilde{w}^2(x)\in K_0[x], \ w(x)=\tilde{w}(x-a), \ v(x)=c \cdot \tilde{w}(x-a).$$
    Hence, in order to prove that $v(x)\in K_0[x]$, it suffices to check that the polynomial
    $\tilde{w}(x)$ lies in $K_0[x]$.
 Let us do it.

 Let  $m:=\deg(\tilde{w})$. If $m=0$, then
    $\tilde{w}(x)=\tilde{w}(0)=1 \in K_0[x]$. Assume now that $m\ge 1$ and
    $$\tilde{w}(x)=1+\sum_{k=1}^m a_k x^k \in K[x],   \ \tilde{w}^2(x)=1+\sum_{k=1}^{2m} b_k x^k \in K_0[x].$$
    We know that all $b_k \in K_0$ and need to prove that all $a_k\in K_0$. We use induction on $k$.   First,
      $b_1=2 a_1$. Since  $\fchar(K)\ne 2$, we have $a_1 \in K_0$, and the first step of induction is done.
     (Notice that we have also proven that $\tilde{w}(x)\in K_0[x]$ if $m\le 1$.)
    Now assume that $k>1$ (and therefore $m \ge k>1$) and $a_i\in K_0$ for all $i <k$. Then
    $$b_k=1\cdot a_k+a_k \cdot 1 + B_k,  \ \text{ where } B_k= \sum_{1\le i,j\le k-1, i+j=k} a_i a_j.$$
    By induction assumption, all $a_i$ and $a_j$ with $1\le i,j\le k-1$ lie in $K_0$. This implies that $B_k \in K_0$.
    Since $b_k=a_k+a_k+B_k$ lies in $K_0$, we have $2a_k \in K_0$ and therefore $a_k \in K_0$. This ends the proof.
\end{proof}

\begin{Rem}
\label{exist2gp1}
Let $K_0$ be a subfield of $K$ and $g$ a positive integer. It follows from Examples \ref{exNotDivide} and \ref{exDivide} that there is a degree $2g+1$ monic polynomial $f(x)\in K_0[x]$ without multiple roots such that the odd degree genus $g$ hyperelliptic curve
$\mathcal{C}_f:y^2=f(x)$ defined over $K_0$ has a torsion point of order $2g+1$ in $\mathcal{C}_f(K_0)$.
\end{Rem}

\begin{thm}\label{PordK02g+1}
Let $K_0$ be a subfield of $K$.
Let $g \ge 1$ be an integer and $f(x)\in K_0[x]$ be  a monic degree  $2g+1$ polynomial without multiple roots.
Suppose that  the odd degree genus $g$  hyperelliptic curve
  $C_f: y^2=f(x)$ over $K_0$ has  $K_0$-points $P=(a_1,c_1)$ and $Q=(a_2,c_2)$ of order $2g+1$ such that  $Q\ne P, \iota(P)$, i.e.,
  $$a_i,c_i \in K_0,\ c_i^2=f(a_i) \ \text{ for } \ i=1,2, \ a_1 \ne a_2.$$
 Then there exists  precisely one ordered pair of polynomials  $u_1(x), u_2(x)\in K_0[x]$ such that the following conditions hold.

 \begin{itemize}
 \item[(i)]
 $\deg(u_i)\le g$ for $i=1,2$.
 \item[(ii)]
$u_1(x)u_2(x)=(x-a_2)^{2g+1}-(x-a_1)^{2g+1}.$
 \item[(iii)]
 If $\fchar(K_0)$ does not divide $2g+1$, then
 $\deg(u_1)=\deg(u_2)=g.$
 \item[(iv)]
 $u_1(a_1)+u_2(a_1) \ne 0, \  u_1(a_2)-u_2(a_2) \ne 0$.
 In particular, $u_2(x)\ne \pm u_1(x)$.
  \item[(v)]
 $$f(x)=(x-a_1)^{2g+1}+\left(\frac{u_1(x)+u_2(x)}{2}\right)^2=(x-a_2)^{2g+1}+\left(\frac{u_1(x)-u_2(x)}{2}\right)^2.$$
 \item[(vi)]
 $$P=\left(a_1,\frac{u_1(a_1)+u_2(a_1)}{2}\right),
 \ Q=\left(a_2,\frac{u_1(a_1)-u_2(a_2)}{2}\right).$$
 \end{itemize}
\end{thm}

\begin{proof}
It follows from Theorem \ref{ordK02g+1} that there exists precisely one pair of polynomials
$v_1(x), v_2(x)\in K_0[x]$  such that for $i=1,2$ we have
   $$\deg(v_i)\le g, \ v_i(a_i)\neq0, \ f(x)=(x-a_i)^{2g+1}+v_i^2(x),  \ P_i=(a_i, v_i(a_i)). $$
We obtain
$$0=\left((x-a_2)^{2g+1}+v_2^2(x)\right)-\left((x-a_1)^{2g+1}+v_1^2(x)\right),$$
i.e.,
$$(x-a_2)^{2g+1}-(x-a_1)^{2g+1}=v_1(x)^2 -v_2^2(x)=(v_1(x)+v_2(x))(v_1(x)-v_2(x)).$$
Let us put
$$u_1(x):=v_1(x)+v_2(x), \  u_2(x):=v_1(x)-v_2(x).$$
Then
 $$u_1(x)u_2(x)=(x-a_2)^{2g+1}-(x-a_1)^{2g+1},$$
 which gives us (ii).
 Clearly,
$$v_1(x)=\frac{u_1(x)+u_2(x)}{2}, \  v_2(x)=\frac{u_1(x)-u_2(x)}{2}.$$
This implies that
$$u_1(a_1)+u_2(a_1) \ne 0, \  u_1(a_1)-u_2(a_1) \ne 0, \ \deg(u_i)\le g \
\text{ for}\ i=1,2,$$
which gives us (iv) and (i), and
 $$f(x)=(x-a_1)^{2g+1}+\left(\frac{u_1(x)+u_2(x)}{2}\right)^2=(x-a_2)^{2g+1}+\left(\frac{u_1(x)-u_2(x)}{2}\right)^2,$$
 which gives us (v).

 We have
 $$\begin{aligned} P=(a_1, v_1(a_1))=\left(a_1,\frac{u_1(a_1)+u_2(a_1)}{2}\right),\\
Q=(a_2, v_2(a_2))=\left(a_2,\frac{u_1(a_2)-u_2(a_2)}{2}\right),\end{aligned}$$
which gives us (vi).

If $\fchar(K_0)$ does not divide $2g+1$, then the polynomial $(x-a_2)^{2g+1}-(x-a_1)^{2g+1}$ has degree $2g$ (and leading coefficient $(2g+1)(a_1-a_2)$), and therefore
$$2g=\deg(u_1)+\deg(u_2).$$
Since both $\deg(u_1),\deg(u_2)\le g$, we conclude that
 $$\deg(u_1)=\deg(u_2)=g,$$
 which gives us (iii).

 It remains to prove the uniqueness of $u_1(x),u_2(x)$. It follows from (v)   that both polynomials
 $u_1(x)+u_2(x)$ and $u_1(x)-u_2(x)$ are defined up to sign. However, (iv) and (vi) determine  $u_1(x)+u_2(x)$ and $u_1(x)-u_2(x)$  uniquely. This implies the uniqueness of  $u_1(x),u_2(x)$.
\end{proof}

\begin{Rem}
\label{diff}
Let $a_1, a_2$ be distinct elements of $K$.
Let
$p:=\fchar(K)$,
  and let $x_0\in K$ be a root of $(x-a_2)^{2g+1}-(x-a_1)^{2g+1}$
Since $a_1 \ne a_2$,  we get $x_0 \ne a_1$ and $x_0 \ne 0$, i.e.,
$$(x_0-a_2)^{2g}\ne 0, \ (x_0-a_1)^{2g}\ne 0.$$
 We have
 $$\left((x-a_2)^{2g+1}-(x-a_1)^{2g+1}\right)^{\prime}=(2g+1)(x-a_2)^{2g}-(2g+1)(x-a_1)^{2g}=$$
 $$(2g+1)\left((x-a_2)^{2g}-(x-a_1)^{2g}\right).$$
 In particular,  if $p$ divides $2g+1$, then $p>2$,  $$\left((x-a_2)^{2g+1}-(x-a_1)^{2g+1}\right)^{\prime}=0,$$ and
 $$(x-a_2)^{2g+1}-(x-a_1)^{2g+1}=\left((x-a_2)^{(2g+1)/p}-(x-a_1)^{(2g+1)/p}\right)^{p};$$
 in particular, all roots of $(x-a_2)^{2g+1}-(x-a_1)^{2g+1}$, including $x_0$, are multiple.
 Now suppose that $\fchar(K)$ does not divide  $2g+1$.  Then
$$\left((x-a_2)^{2g+1}-(x-a_1)^{2g+1}\right)^{\prime}\ne 0.$$
 Assume additionally that $x_0$ is a multiple root of $(x-a_2)^{2g+1}-(x-a_1)^{2g+1}$. This means that
 $$(x_0-a_2)^{2g+1}=(x_0-a_1)^{2g+1}, \ (x_0-a_2)^{2g}=(x_0-a_1)^{2g}.$$
 Dividing the first equality by the second one, we get
  $$x_0-a_2 =x_0-a_1,$$
  and therefore $a_1=a_2$, which is not the case. The obtained contradiction proves that if  $\fchar(K)$ does not divide $2g+1$, then the polynomial
  $(x-a_2)^{2g+1}-(x-a_1)^{2g+1}$ has no multiple roots.
 \end{Rem}

\begin{thm}\label{ordKf2g+1}
Let $K_0$ be a subfield of $K$ and  $g \ge 1$ be an integer.  Let $a_1$ and $a_2$ be distinct elements of $K_0$. Let $u_1(x), u_2(x)\in K_0[x]$ be polynomials such that
$ \deg(u_i)\le g \ \text{ for } \ i=1,2$ and  $$\ u_1(x)u_2(x)=(x-a_2)^{2g+1}-(x-a_1)^{2g+1}.$$
Assume additionally that if $\fchar(K)$ does not divide $2g+1$, then
  $\deg(u_1)=\deg(u_2)=~g.$
 Let us consider the monic degree $2g+1$ polynomial
 $$f_{a_1,a_2;u_1,u_2}(x)=(x-a_1)^{2g+1}+\left(\frac{u_1(x)+u_2(x)}{2}\right)^2.$$
 Then the following conditions hold.

 \begin{itemize}
 \item[(a)]
 $$f_{a_1,a_2;u_1,u_2}(x)=(x-a_2)^{2g+1}+\left(\frac{u_1(x)-u_2(x)}{2}\right)^2=f_{a_2,a_1;u_1,-u_2}(x).$$
 \item[(b)]
 $$f_{a_1,a_2;u_1,u_2}(x+a_1)=f_{0,a;\tilde{u}_1,\tilde{u}_2}(x)=x^{2g+1}+\left(\frac{\tilde{u}_1(x)+\tilde{u}_2(x)}{2}\right)^2,$$
 where
 $$a:=a_2-a_1\in K^{*}_0, \ \tilde{u}_1(x):=u_1(x+a_1)\in K_0[x],\ \tilde{u}_2(x)=u_2(x+a_1)\in K_0[x].$$
 In addition,
 $$  \tilde{u}_1(x) \tilde{u}_2(x)=(x-a)^{2g+1}-x^{2g+1}.$$

 \item[(c)] Suppose that $f_{a_1,a_2;u_1,u_2}(x)$ has no multiple roots. Then
 \begin{enumerate}
 \item[(c1)]
 $u_1^{\prime}(x)\ne 0, \ u_2^{\prime}(x) \ne 0.$
 In particular, neither $u_1(x)$ nor $u_2(x)$ is a constant.
  \item[(c2)]
 Let us consider the odd degree
 genus $g$ hyperelliptic curve
 $$\mathcal{C}_{a_1,a_2;u_1,u_2}:=\mathcal{C}_{f_{a_1,a_2;u_1,u_2}}: y^2=f_{a_1,a_2;u_1,u_2}(x),$$
 which  is defined over $K_0$.
 Then
  $$P_{a_1,a_2;u_1,u_2}=\left(a_1,\frac{u_1(a_1)+u_2(a_1)}{2}\right)\ \text{ and }
 \ Q_{a_1,a_2;u_1,u_2}=\left(a_2,\frac{u_1(a_2)-u_2(a_2)}{2}\right)$$
 are points of order $2g+1$ in $\mathcal{C}_{a_1,a_2;u_1,u_2}(K_0)$.
 \end{enumerate}
 \end{itemize}
 \end{thm}

 \begin{proof}
 $$f_{a_1,a_2;u_1,u_2}(x)-\left((x-a_2)^{2g+1}+\left(\frac{u_1(x)-u_2(x)}{2}\right)^2\right)$$
$$= (x-a_1)^{2g+1}+\left(\frac{u_1(x)+u_2(x)}{2}\right)^2-(x-a_2)^{2g+1}-\left(\frac{u_1(x)-u_2(x)}{2}\right)^2$$
 $$=\left(\frac{u_1(x)+u_2(x)}{2}\right)^2-\left(\frac{u_1(x)-u_2(x)}{2}\right)^2
 +\left((x-a_1)^{2g+1}-(x-a_2)^{2g+1}\right)$$
 $$=u_1(x)u_2(x)-u_1(x)u_2(x)=0.$$
 This proves (a).

 Let us prove (b). Clearly,  $\deg(u(x+a_1))=\deg(u(x))$ for every polynomial $u(x)\in K[x]$. This implies that
 $\deg(\tilde{u}_1)=\deg(u_1), \deg(\tilde{u}_2)=\deg(u_2)$. It follows that  $\deg(\tilde{u}_1)=\deg(\tilde{u}_2)=g$
 if $\deg(u_1)=\deg(u_2)=g$. We have
 $$\begin{aligned}(x-a)^{2g+1}-x^{2g+1}&=( (x+a_1)-a_2)^{2g+1}-((x+a_1)-a_1)^{2g+1}\\&=u_1(x+a_1) u_2(x+a_1)=\tilde{u}_1(x) \tilde{u}_2(x).\end{aligned}$$
Finally,
  $$\begin{aligned}f_{a_1,a_2;u_1,u_2}(x+a_1)&=
  ((x-a_1)+a_1)^{2g+1}+\left(\frac{u_1(x+a_1)+u_2(x+a_1)}{2}\right)^2\\&=
 (x-0)^{2g+1}+
  \left(\frac{\tilde{u}_1(x)+\tilde{u}_2(x)}{2}\right)^2=f_{0,a;\tilde{u}_1,\tilde{u}_2}(x).\end{aligned}$$

 Let us prove (c1).  We put
$p:=\fchar(K).$
Assume that, say, $u_1^{\prime}(x)=0$. We need to arrive to a contradiction. Under our assumption one of the following condition holds.
 \begin{enumerate}
 \item[(i)]
  $u_1 (x)$ is a nonzero constant, i.e., $\deg(u_1)=0<g$.
 This implies that
  $\fchar(K)$ is a prime dividing $2g+1$.
  \item[(ii)]
  $p$ is a prime and there exists a polynomial $w_1(x)\in K[x]$ such that $ u_1(x)=w_1^p(x)$.
  \end{enumerate}

  Clearly, in both cases $p$ is a prime dividing $2g+1$,
  and there exists a polynomial $w_1(x)\in K[x]$ such that $ u_1(x)=w_1^p(x)$.  We have
  $$w_1^p(x) u_2(x)=u_1(x)u_2(x)=(x-a_2)^{2g+1}-(x-a_1)^{2g+1}=$$
  $$\left((x-a_2)^{(2g+1)/p}-(x-a_1)^{(2g+1)/p}\right)^{p}.$$
It follows that $w_1(x)$ divides $(x-a_2)^{(2g+1)/p}-(x-a_1)^{(2g+1)/p}$ in $K[x]$, i.e., there exists a polynomial $w_1(x)\in K[x]$ such that
  $$w_1(x)w_2(x)=(x-a_2)^{(2g+1)/p}-(x-a_1)^{(2g+1)/p},$$
  and therefore
  $$(x-a_2)^{2g+1}-(x-a_1)^{2g+1}=\left(w_1(x)w_2(x)\right)^{p}=w_1^p(x) w_2^p(x)=u_1(x) w_2^p(x).$$
 We have $u_2(x)=w_2^p(x).$ Consequently,
  $$\begin{aligned}f_{a_1,a_2;u_1,u_2}(x)=(x-a_1)^{2g+1}+\left(\frac{w_1^p(x)+w_2^p(x)}{2}\right)^2\\=
  \left((x-a_1)^{(2g+1)/p}+\left(\frac{w_1(x)+w_2(x)}{\sqrt[p]{2}}\right)^2\right)^p.\end{aligned}$$
  Hence $f_{a_1,a_2;u_1,u_2}(x)$ is a $p$th power in $K[x]$, and therefore all its roots are multiple,
  which contradicts our assumptions. Therefore, $u_1^{\prime}(x)\ne 0$. By the same token, $u_2^{\prime}(x)\ne 0$.
  This ends the proof of (c1).

 In order to prove (c2), notice that, from the very definition of $f_{a_1,a_2;u_1,u_2}(x)$, it follows that $P_{a_1,a_2;u_1,u_2}$  lies on
 $\mathcal{C}_{a_1,a_2;u_1,u_2}$. The fact that $Q_{a_1,a_2;u_1,u_2}$ lies on $\mathcal{C}_{a_1,a_2;u_1,u_2}$ follows from (a).  Applying  two times Theorem \ref{ordK02g+1} to $a=a_1, v(x)=(u_1(x)+u_2(x))/2$ and to $a=a_2, v(x)=(u_1(x)-u_2(x))/2$, we conclude that both $P_{a_1,a_2;u_1,u_2}$ and $Q_{a_1,a_2;u_1,u_2}$ are points of order $2g+1$ in
 $\mathcal{C}_{a_1,a_2;u_1,u_2}(K_0)\subset \mathcal{C}_{a_1,a_2;u_1,u_2}(K)$. In addition,
 $$\frac{u_1(a_1)+u_2(a_1)}{2} \ne 0, \  \frac{u_1(a_2)-u_2(a_2)}{2} \ne 0,$$
 i.e.,
 $u_1(a_1)+u_2(a_1) \ne 0, \  u_1(a_2)-u_2(a_2) \ne 0$.
 \end{proof}

 \begin{Rem}
 \label{produ1u2}
 Let $a_1,a_2$ be distinct elements of  a subfield $K_0\subset K$ and let
 $u_1(x), u_2(x)\in K_0[x]$ be polynomials that satisfy $$u_1(x)u_2(x)=(x-a_2)^{2g+1}-(x-a_1)^{2g+1}.$$ Then
 $$u_1(a_1)u_2(a_1)=(a_1-a_2)^{2g+1}-(a_1-a_1)^{2g+1}=(a_1-a_2)^{2g+1} \ne 0,$$
 $$u_1(a_2)u_2(a_2)=(a_2-a_2)^{2g+1}-(a_2-a_1)^{2g+1}=(a_1-a_2)^{2g+1} \ne 0.$$
 In particular,
 $$u_1(a_1)\ne 0, \ u_2(a_1)\ne 0,\ u_1(a_2)\ne 0,\ u_2(a_2)\ne 0.$$

 \end{Rem}

 \begin{Rem}\label{changesign}
 Let $a_1,a_2$ be distinct elements of  a subfield $K_0\subset K$, and let
 $u_1(x), u_2(x)\in K_0[x]$ be polynomials that satisfy $$u_1(x)u_2(x)=(x-a_2)^{2g+1}-(x-a_1)^{2g+1}.$$ Then $-u_1(x),-u_2(x) \in K_0[x]$ and
 $$\begin{aligned} (x-a_2)^{2g+1}-(x-a_1)^{2g+1}=(-u_1(x))(-u_2(x))\\=u_2(x)u_1(x)=(-u_2(x))(-u_1(x)).\end{aligned}$$
  Assume additionally that $\deg(u_1)\le g, \deg(u_2) \le g$, and the equalities hold if $\fchar(K)$ does not divide $2g+1$. Then
 $$f_{a_1,a_2;u_1,u_2}(x)=f_{a_1,a_2;-u_1,-u_2}(x)=f_{a_1,a_2;u_2,u_1}(x)=f_{a_1,a_2;-u_2,-u_1}(x).$$
 If, in addition, $f_{a_1,a_2;u_1,u_2}(x)$ has no multiple roots, then
 $$\mathcal{C}_{a_1,a_2;u_1,u_2}=\mathcal{C}_{a_1,a_2;-u_1,-u_2}=\mathcal{C}_{a_1,a_2;u_2,u_1}=\mathcal{C}_{a_1,a_2;-u_2,-u_1}.$$
 So, in all four cases we get the same odd degree hyperelliptic curve.
 However, it follows readily from Theorem \ref{ordKf2g+1}(c1) that
 $$P_{a_1,a_2;-u_1,-u_2}=\iota(P_{a_1,a_2;u_1,u_2}), \
 Q_{a_1,a_2;-u_1,-u_2}=\iota(Q_{a_1,a_2;u_1,u_2}),$$
 $$P_{a_1,a_2;u_2,u_1}=P_{a_1,a_2;u_1,u_2}, \ Q_{a_1,a_2;u_2,u_1}=\iota(Q_{a_1,a_2;u_1,u_2}),$$
  $$P_{a_1,a_2;-u_2,-u_1}=\iota(P_{a_1,a_2;u_1,u_2}), \ Q_{a_1,a_2;u_2,u_1}=Q_{a_1,a_2;u_1,u_2}.$$
 \end{Rem}

 \begin{Rem}\label{unique}
 Let $a_1,a_2$ be distinct elements of  a subfield $K_0\subset K$, and let
 $u_1(x)$, $u_2(x)$, $\tilde{u}_1(x)$, $\tilde{u}_2(x)\in K_0[x]$ be polynomials that satisfy
 $$u_1(x)u_2(x)=(x-a_2)^{2g+1}-(x-a_1)^{2g+1}=\tilde{u}_1(x) \tilde{u}_2(x).$$
  Let us assume  that $\deg(u_1)\le g, \deg(u_2) \le g$. In addition, we also assume that the equalities hold if $\fchar(K)$ does not divide $2g+1$.

 Suppose that
 $$f_{a_1,a_2;u_1,u_2}(x)=f_{a_1,a_2;\tilde{u}_1,\tilde{u}_2}(x),$$
  i.e.,
 $$(x-a_1)^{2g+1}+\left(\frac{u_1(x)+u_2(x)}{2}\right)^2=(x-a_1)^{2g+1}+\left(\frac{\tilde{u}_1(x)+\tilde{u}_2(x)}{2}\right)^2.$$
 This means that
 $$\left(\frac{u_1(x)+u_2(x)}{2}\right)^2=\left(\frac{\tilde{u}_1(x)+\tilde{u}_2(x)}{2}\right)^2,$$
 i.e.,
 $$\tilde{u}_1(x)+\tilde{u}_2(x)= \pm \left(u_1(x)+u_2(x)\right).$$
 Since
 $$u_1(x)u_2(x)=\tilde{u}_1(x) \tilde{u}_2(x)=(-u_1(x)(-u_2(x)),$$
 we conclude that  one of the following four conditions holds.
 \begin{itemize}
 \item
$ \tilde{u}_1(x)=u_1(x), \tilde{u}_2(x)=u_2(x)$;
 \item
$ \tilde{u}_1(x)=-u_1(x), \tilde{u}_2(x)=-u_2(x)$;
\item
$ \tilde{u}_1(x)=u_2(x), \tilde{u}_2(x)=u_1(x)$;
 \item
$ \tilde{u}_1(x)=-u_2(x), \tilde{u}_2(x)=-u_1(x)$.
 \end{itemize}
 \end{Rem}

 \begin{thm}\label{char2gPlus1}
 Let $p=\fchar(K)$ be an odd prime, and let $g$ be a positive integer such that $2g+1=p^k$  for a positive integer $k$. (E.g., $g=(p-1)/2$.)
 Let  $f(x)\in K[x]$ be  a monic degree  $2g+1$ polynomial without multiple roots and $\mathcal{C}_f: y^2=f(x)$ be the corresponding odd degree genus $g$ hyperelliptic curve.
Then $\mathcal{C}_f(K)$ contains  at most  two points of order $p^k$.
 \end{thm}

 \begin{proof}
 Assume that $\mathcal{C}_f(K)$ contains at least three points of order $p^k=2g+1$. Let $P \in \mathcal{C}_f(K)$ be one of them. Then $P=(a_1,c_1)$ with
 $$a_1,c_1 \in K, c_1 \ne 0,  \ c_1^2=f(a_1).$$
Consequently, $\iota(P)=(a_1,-c_1)\in \mathcal{C}_f(K)$  also has order $2g+1$. Hence there exists another point $Q \in \mathcal{C}_f(K)$ of order $2g+1$ that is neither $P$ nor $\iota(P)$. This implies that $Q=(a_2,c_2)$ with
 $$a_2,c_2 \in K, c_2 \ne 0,  \ c_2^2=f(a_2), \ a_2 \ne a_1.$$
 By Theorem \ref{PordK02g+1} (applied to $K_0=K$), there exist polynomials $u_1(x),u_2(x)\in K[x]$ such that
 $$u_1(x)u_2(x)=(x-a_2)^{2g+1}-(x-a_1)^{2g+1}, \ f(x)=(x-a_1)^{2g+1}+\left(\frac{u_1(x)+u_2(x)}{2}\right)^2.$$
 Since $2g+1=p^k$ and $p=\fchar(K)$, the difference
 $$(x-a_2)^{2g+1}-(x-a_1)^{2g+1}=(x-a_2)^{p^k}-(x-a_1)^{p^k}=(a_1-a_2)^{p^k}$$
 is a nonzero element of $K$. It follows that both $u_1(x)$ and $u_2(x)$ are also nonzero elements of $K$,
 $u_1(x)=b_1\in K^{*}, \ u_2(x)=b_2 \in K^{*}$.
 Consequently,
 $$f(x)=(x-a_1)^{p^k}+\left(\frac{b_1+b_2}{2}\right)^2=
 (x-a_1+b)^{p^k},$$
 where
 $$b=  \left( \sqrt[p^k]{\frac{b_1+b_2}{2}}\right)^2.$$
 Therefore, $f(x)$ has multiple roots, which gives us the desired contradiction.
 \end{proof}

   \begin{Rem}
   The case $p=5, g=2, k=1$ of Theorem \ref{char2gPlus1} was done in \cite[Lemma 3.1]{Box1}.
   \end{Rem}

   \begin{Rem}
   \label{ell3}
   Let us consider the case when $p=\fchar(K)=3$ and $f(x)$ is a degree $3$ polynomial without multiple roots. Then the equation
   $y^2=f(x)$ defines an elliptic curve over the field $K$ of characteristic $3$. It is well known that an elliptic curve in characteristic 3
   has at most two points of order 3.  Theorem \ref{char2gPlus1} may be viewed as a generalization of this fact, where $3=3^1$ is replaced by any odd prime $p$ and $1$ by any positive integer $k$.
     \end{Rem}

     \section{Families of hyperelliptic curves}
     \label{famCurves}

 \begin{thm}
 \label{Pdegeneration}
 Let us assume that $\fchar(K)$ does not divide $2g+1$.
 Let $w_1(x), w_2(x)\in K[x]$ be degree $g$ polynomials without common roots.
 Then for all but finitely many $\lambda\in K^{*}$ the degree $2g+1$ polynomial
 $$h_{\lambda}(x)=\lambda x^{2g+1}+\left(\lambda w_1(x)+w_2(x)\right)^2$$
 has no multiple roots.
 \end{thm}

 \begin{proof}
 Fix  $x_0 \in K$. Then
 $$h_{\lambda}(x_0)=w_1^2(x_0)\lambda^2+(x_0^{2g+1}+2w_1(x_0)w_2(x_0))\lambda +w_2(x_0)^2$$
 is a polynomial in $\lambda$ of degree $\le 2$ such that at least one of its coefficients does not vanish. Indeed, either  its coefficient $w_1^2(x_0)$
 at $\lambda^2$ is not $0$ or its constant term $w_2(x_0)^2$ does not vanish,
 because either $w_1(x_0)\ne 0$ or $w_2(x_0)\ne 0$. This implies that there exist at most two $\lambda\in K$ such that  $h_{\lambda}(x_0)=0$.
 Hence, in order to prove the theorem, it suffices to check that there are only finitely many $x_0\in K$ for which there is $\lambda\in K^{*}$ such that
   $h_{\lambda}(x_0)=h^{\prime}_{\lambda}(x_0)=0$.  Our plan is to produce several polynomials in $x$ that do not depend on $\lambda$ and such that our $x_0$ is a root of one of them.

 We have
 $$h_{\lambda}^{\prime}(x)=(2g+1)\lambda  x^{2g}+2\left(\lambda w_1(x)+w_2(x)\right) \left(\lambda w^{\prime}_1(x)+w^{\prime}_2(x)\right).$$
 Suppose that $x_0\in K$ and $\lambda\in K^{*}$ satisfy $h_{\lambda}(x_0)=h_{\lambda}^{\prime}(x)=0$, i.e., $x_0$ is a multiple root of $h_{\lambda}(x)$.
 This means that $x_0$ is a solution of the system
 $$\lambda x^{2g+1}+\left(\lambda w_1(x)+w_2(x)\right)^2=0,$$
 $$(2g+1)\lambda  x^{2g}+2\left(\lambda w_1(x)+w_2(x)\right) \left(\lambda w^{\prime}_1(x)+w^{\prime}_2(x)\right)=0.$$
 Multiplying the second equation  by $x$ and the first equation  by $2g+1$, and subtracting one from the other, we obtain that $x_0$ is a solution of the
  equation
 $$(2g+1)\left(\lambda w_1(x)+w_2(x)\right)^2-2x\left(\lambda w_1(x)+w_2(x)\right) \left(\lambda w^{\prime}_1(x)+w^{\prime}_2(x)\right)=0.$$
Hence either
 \begin{itemize}
 \item[(i)]
 $\lambda w_1(x_0)+w_2(x_0)=0$

\noindent or
  \item[(ii)]  $(2g+1)\left(\lambda w_1(x_0)+w_2(x_0)\right)-2x_0\left(\lambda w^{\prime}_1(x_0)+w^{\prime}_2(x_0)\right)=0$.
 \end{itemize}

 Case (i).  Since the set of roots of $w_1(x)$ is finite, we may assume that $x_0$ is not one of them and get
 $\lambda=-w_2(x_0)/w_1(x_0)$. It follows from the first equation of the system that $x_0$ is a solution of the equation
 $$-\frac{w_2(x)}{w_1(x)} x^{2g+1}+\left(-\frac{w_2(x)}{w_1(x)} w_1(x)+w_2(x)\right)^2=0.$$
 This means that $-w_2(x_0)x_0^{2g+1}/ w_1(x_0) =0$,
which implies that  case (i) holds only for finitely many values of $x_0$, namely if either $x_0=0$  or $x_0$ is one of the  roots of $w_2(x)$.

 Case (ii). In this case we have
 $$\left((2g+1)w_1(x_0)-2x_0 w^{\prime}_1(x_0)\right)\lambda=2 x_0 w^{\prime}_2(x_0)-(2g+1)w_2(x_0).$$
 Since   $\deg(w_1)=g \ne (2g+1)/2$, the polynomial $(2g+1)w_1(x_0)-2x_0 w^{\prime}_1(x)$ has degree $g$ and the set of its roots is finite.
 So, we may assume that $x_0$ is not one of them, i.e., $ (2g+1)w_1(x_0)-2x_0 w^{\prime}_1(x_0) \ne 0$ and
 $$\lambda=\frac{2 x_0 w^{\prime}_2(x_0)-(2g+1)w_2(x_0) }{ (2g+1)w_1(x_0)-2x_0 w^{\prime}_1(x_0) }.$$
 Plugging  this expression for $\lambda$ in the first equation of the system, we get that $x_0$ is a solution of the equation
 $$\begin{aligned}&\frac{2 x w^{\prime}_2(x)-(2g+1)w_2(x) }{(2g+1)w_1(x)-2x w^{\prime}_1(x)} x^{2g+1}\\+&\left(\frac{2 x w^{\prime}_2(x)-(2g+1)w_2(x) }{ (2g+1)w_1(x)-2x w^{\prime}_1(x) } w_1(x)+w_2(x)\right)^2=0.\end{aligned}$$
 This means that $x_0$ is a root of the polynomial
 $$\begin{aligned}&H(x):=(2 x w^{\prime}_2(x)-(2g+1)w_2(x) )((2g+1)w_1(x)-2x w^{\prime}_1(x)) x^{2g+1}\\&+
\left((2 x w^{\prime}_2(x)-(2g+1)w_2(x) ) w_1(x)+((2g+1)w_1(x)-2x w^{\prime}_1(x))w_2(x)\right)^2.\end{aligned}$$
Since $\deg(w_1)=\deg(w_2)=g \ne (2g+1)/2$, both polynomials $(2 x w^{\prime}_2(x)-(2g+1)w_2(x) )$ and $((2g+1)w_1(x)-2x w^{\prime}_1(x))$
have degree $g$. This implies that the first term in the formula for $H(x)$ is a polynomial of degree $g+g+(2g+1)=4g+1$. On the other hand,
the second term in the formula for $H(x)$ is a polynomial of degree $\le 2\cdot (g+g)=4g$. Therefore, $\deg(H)=4g+1$ and the set of roots of $H(x)$ is finite.

To summarize: there are only finitely many $x_0\in K$ such that there exists $\lambda \in K^{*}$ for which $x_0$ is a multiple root of $h_{\lambda}(x)$. This ends the proof.
 \end{proof}

 \begin{thm}
 \label{degeneration}
 Let us assume that $\fchar(K)$ does not divide $2g+1$.
 Let $a_1,a_2$ be distinct elements of  $ K$, and let
 $u_1(x), u_2(x)\in K[x]$ be degree $g$ polynomials that satisfy
 $$  u_1(x)u_2(x)=(x-a_2)^{2g+1}-(x-a_1)^{2g+1}.$$

  Then the following conditions hold.

 \begin{itemize}
 \item[(i)]
 If $\mu \in K^{*}$, then $\mu u_1(x), \mu^{-1} u_2(x) \in K[x]$ are degree $g$ polynomials that satisfy
 $$(\mu u_1(x)) (\mu^{-1} u_2(x))= u_1(x)u_2(x)=(x-a_2)^{2g+1}-(x-a_1)^{2g+1}.$$
  \item[(ii)]
 There are only finitely many $\mu \in K^{*}$ such that the polynomial
 $$f_{a_1,a_2;\mu u_1,\mu^{-1}u_2}(x)=(x-a_1)^{2g+1}+\left(\frac{\mu u_1(x)+\mu^{-1}u_2(x)}{2}\right)^2$$
 has a multiple root.
 \end{itemize}
 \end{thm}

 \begin{proof}
 Using  Theorem \ref{ordKf2g+1}(b), we may and will assume that $a_1=0$, $a_2=a \ne 0$, and
 $$f_{a_1,a_2;\mu u_1,\mu^{-1}u_2}(x)=f_{0,a;\mu u_1,\mu^{-1}u_2}(x).$$
 We have
 $$u_1(x)u_2(x)=(x-a)^{2g+1}-x^{2g+1}$$
 and
 $$u_i(0) \ne 0, u_i(a) \ne 0 \ \text{ for } \ i=1,2.$$
 Since $\fchar(K)$ does not divide $2g+1$, Remark \ref{diff} tells us that the polynomial $(x-a)^{2g+1}-x^{2g+1}$ has no multiple roots.
 This implies that $u_1(x)$ and $u_2(x)$ have no common roots.
 We have
 $$\begin{aligned}f_{0,a;\mu u_1,\mu^{-1}u_2}(x)&=
 x^{2g+1}+\left(\frac{\mu u_1(x)+\mu^{-1}u_2(x)}{2}\right)^2\\&=x^{2g+1}+\left(\mu w_1(x)+\mu^{-1}w_2(x)\right)^2,\end{aligned}$$
 where $w_1(x)=u_1(x)/2, \ w_2(x)=u_2(x)/2$.  Clearly, $w_1(x)$ and $w_2(x)$ are degree $g$ polynomials without common roots.
 We have
 $$\mu^2 f_{0,a;\mu u_1,\mu^{-1}u_2}(x)=\mu^2 x^{2g+1}+\left(\mu^2 w_1(x)+w_2(x)\right)^2.$$
 It follows from Theorem \ref{Pdegeneration} that  there is a finite set $S \subset K^{*}$ such that if $\mu^2 \not\in S$, then
 $\mu^2 f_{a_1,a_2;\mu u_1,\mu^{-1}u_2}(x)$  has no multiple roots and therefore   $f_{0,a;\mu u_1,\mu^{-1}u_2}(x)$
 also has no multiple roots. Consequently, $f_{0,a;\mu u_1,\mu^{-1}u_2}(x)$ has no multiple roots for all but finitely many $\mu\in K^{*}$.
 \end{proof}

 \begin{thm}
 \label{pdegeneration2}
 Let us assume that $p:=\fchar(K)>0$, $p$  divides $2g+1$, but $2g+1$ is not a power of $p$.
 Let $w_1(x), w_2(x)\in K[x]$ be nonconstant polynomials such that
 $$\deg(w_1)\le g, \deg(w_2)\le g; w_1^{\prime}(x)\ne 0, w_2^{\prime}(x)\ne 0;  w_1(0)\ne 0, w_2(0) \ne 0.$$
 Assume also that
 $$\left(w_1(x)w_2(x)\right)^{\prime}=0.$$
 Then for all but finitely many $\lambda\in K^{*}$ the degree $2g+1$ polynomial
 $$h_{\lambda}(x)=\lambda x^{2g+1}+\left(\lambda w_1(x)+w_2(x)\right)^2$$
 has no multiple roots.
 \end{thm}

 \begin{proof}
 Fix  $x_0 \in K$. Then
 $$h_{\lambda}(x_0)=w_1^2(x_0)\lambda^2+(x_0^{2g+1}+2w_1(x_0)w_2(x_0))\lambda +w_2(x_0)^2$$
 is a polynomial in $\lambda$ of degree $\le 2$ such that at least one of its coefficients does not vanish. Indeed, if
 $$w_1^2(x_0)=0,\ w_2(x_0)^2=0, \ x_0^{2g+1}+2w_1(x_0)w_2(x_0)=0,$$
then $$w_1(x_0)=0, \ w_2(x_0)=0, \ x_0=0,$$
 which means that
 $$x_0=0,\ w_1(0)=0,\ w_2(0)=0.$$
 However, $x_0=0$ is not a zero of $w_1(x)$, which gives us the desired contradiction.

  This implies that for any given $x_0\in K$ there exist at most two $\lambda\in K$ such that $h_{\lambda}(x_0)=0$.
 Hence, in order to prove the theorem, it suffices to check that there are only finitely many $x_0\in K$ for which there is $\lambda\in K^{*}$ such that
 $h_{\lambda}(x_0)=h_{\lambda}^{\prime}(x_0)=0$.  Our plan is to produce (as in the proof of Theorem \ref{Pdegeneration}) several polynomials in $x$ that do not depend on $\lambda$ and such that our $x_0$ is a root of one of them. From the very beginning, we may exclude finally many values of $x_0$.  In particular, we may and will assume that
 \beq\label{notx0}
 x_0 \ne 0, \ w_1(x_0)\ne 0,\  w_1^{\prime}(x_0)\ne 0,\ w_2(x_0)\ne 0,\  w_2^{\prime}(x_0)\ne 0.
 \eeq
 Since the derivative of $w_1(x)w_2(x)$ is identically $0$, we get
 $$0=w_1^{\prime}(x_0)w_2(x_0)+w_2^{\prime}(x_0)w_1(x_0)$$
 and therefore
 \beq\label{ratioDER}
 \frac{w_2^{\prime}(x_0)}{w_1^{\prime}(x_0)}=-\frac{w_2(x_0)}{w_1(x_0)}.
 \eeq
We have
 $$\begin{aligned}h_{\lambda}^{\prime}(x)&=(2g+1)\lambda  x^{2g+1}+2\left(\lambda w_1(x)+w_2(x)\right) \left(\lambda w^{\prime}_1(x)+w^{\prime}_2(x)\right)\\&=
2\left(\lambda w_1(x)+w_2(x)\right) \left(\lambda w^{\prime}_1(x)+w^{\prime}_2(x)\right)
 .\end{aligned}$$

 Suppose that $x_0\in K$ and $\lambda\in K^{*}$ satisfy $h_{\lambda}(x_0)=h_{\lambda}^{\prime}(x_0)=0$, i.e., $x_0$ is a multiple root of $h_{\lambda}(x)$.
 This means that $x_0$ is a solution of the system
 $$\begin{aligned}\lambda x^{2g+1}+\left(\lambda w_1(x)+w_2(x)\right)^2=0, \\
\left(\lambda w_1(x)+w_2(x)\right) \left(\lambda w^{\prime}_1(x)+w^{\prime}_2(x)\right)=0.\end{aligned}$$
 Hence either

 \begin{itemize}
 \item[(i)]
 $\lambda w_1(x_0)+w_2(x_0)=0$

\noindent or

 \item[(ii)]
$\lambda w^{\prime}_1(x_0)+w^{\prime}_2(x_0)=0.$
 \end{itemize}

  Case (i).  Since $w_1(x_0)\ne 0$, we get
 $\lambda=-w_2(x_0)/w_1(x_0)$. It follows from the first equation of the system that $x_0$ is a solution of the equation
 $$-\frac{w_2(x)}{w_1(x)} x^{2g+1}+\left(-\frac{w_2(x)}{w_1(x)} w_1(x)+w_2(x)\right)^2=0.$$
Consequently, $$-\frac{w_2(x_0)}{w_1(x_0)} x_0^{2g+1}=0,$$ which is not the case, since
 $x_0\ne 0$ and $w_2(x_0)\ne 0$. So,  the case (i) does not occur.

  Case (ii).  Since  $w^{\prime}_1(x_0)\ne 0$,  we get
 $\lambda=-w_2^{\prime}(x_0)/w_1^{\prime}(x_0)$. In light of \eqref{ratioDER},
 $$\lambda=\frac{w_2(x_0)}{w_1(x_0)}.$$
 It follows from the first equation of the system that $x_0$ is a solution of the equation
 $$\frac{w_2(x)}{w_1(x)}x^{2g+1}+\left(\frac{w_2(x)}{w_1(x)} w_1(x)+w_2(x)\right)^2=0,$$
 i.e., $x_0$ is a solution of the equation
 $$\frac{w_2(x)}{w_1(x)}x^{2g+1}+\left(2w_2(x)\right)^2=0.$$
 Multiplying this equation by $w_1(x)$, we obtain that
  $x_0$ is a root of the polynomial
 $$w_2(x) x^{2g+1}+
 4(w_2(x))^2 w_1(x)=w_2(x)\left(x^{2g+1}+4 w_1(x) w_2(x)\right).$$
 Since $w_2(x_0)\ne 0$,  $x_0$ is a root of the polynomial
 $H(x)=x^{2g+1}+4 w_1(x) w_2(x)$.
 Since  both $\deg(w_i)\le g$, we have $\deg(w_1(x)w_2(x))\le 2g<2g+1$, and therefore
 $H(x)$ is a polynomial of degree $2g+1$. In  particular, the set of roots of $H(x)$ is finite.

 To summarize: there are only finitely many $x_0\in K$ for which there exists $\lambda \in K^{*}$ such that $x_0$ is a multiple root of $h_{\lambda}(x)$. This ends the proof.
 \end{proof}

 \begin{thm}
 \label{degenerationCHARp}
 Let us assume that $p:=\fchar(K)>0$  and $p$ divides $2g+1$, but $2g+1$ is not a power of $p$.
 Let $a_1,a_2$ be distinct elements of  $ K$, and let
 $u_1(x), u_2(x)\in K[x]$ be  polynomials that satisfy
 $$  u_1(x)u_2(x)=(x-a_2)^{2g+1}-(x-a_1)^{2g+1},$$
 $$\deg(u_1)\le g, \deg(u_2)\le g, \ u_1^{\prime}(x)\ne 0, u_2^{\prime}(x)\ne 0.$$

 Then the following conditions hold.

 \begin{itemize}
 \item[(i)]
 If $\mu \in K^{*}$, then $\mu u_1(x), \mu^{-1} u_2(x) \in K[x]$ are  polynomials of degree $\le g$ such that $$(\mu u_1(x))^{\prime}\ne 0,\ (\mu u_2(x))^{\prime}\ne 0,$$
 $$(\mu u_1(x)) (\mu^{-1} u_2(x))= u_1(x)u_2(x)=(x-a_2)^{2g+1}-(x-a_1)^{2g+1}.$$
  \item[(ii)]
 There are only finitely many $\mu \in K^{*}$ such that the polynomial
 $$f_{a_1,a_2;\mu u_1,\mu^{-1}u_2}(x)=(x-a_1)^{2g+1}+\left(\frac{\mu u_1(x)+\mu^{-1}u_2(x)}{2}\right)^2$$
 has a multiple root.
 \end{itemize}
 \end{thm}

 \begin{proof}
 (i) is obvious. Let us prove (ii).
 Using  Theorem \ref{ordKf2g+1}(b), we may and will assume that $a_1=0$, $a_2=a \ne 0$,
 $$f_{a_1,a_2;\mu u_1,\mu^{-1}u_2}(x)=f_{0,a;\mu u_1,\mu^{-1}u_2}(x),$$
 $$u_1(x)u_2(x)=(x-a)^{2g+1}-x^{2g+1},$$
 and
 $$u_i(0) \ne 0, \ u_i(a) \ne 0 \ \text{ for } \ i=1,2.$$
 Since $\fchar(K)$ divides $2g+1$,  the derivatives of both $(x-a)^{2g+1}$ and $x^{2g+1}$ are $0$. This implies that
 $$(u_1(x)u_2(x))^{\prime}=0.$$
 We have
 $$\begin{aligned}f_{0,a;\mu u_1,\mu^{-1}u_2}(x)&=
 x^{2g+1}+\left(\frac{\mu u_1(x)+\mu^{-1}u_2(x)}{2}\right)^2\\&=x^{2g+1}+\left(\mu w_1(x)+\mu^{-1}w_2(x)\right)^2,\end{aligned}$$
 where $w_1(x)=u_1(x)/2, \ w_2(x)=u_2(x)/2$.  Clearly, $w_1(x)$ and $w_2(x)$ are polynomials of degree $\le g$ and
 $$w_1^{\prime}(x)\ne 0,\ w_2^{\prime}(x)\ne 0,\ (w_1(x)w_2(x))^{\prime}=0.$$
 Since
 $$\mu^2 f_{0,a;\mu u_1,\mu^{-1}u_2}(x) =\mu^2 x^{2g+1}+\left(\mu^2 w_1(x)+w_2(x)\right)^2,$$
 it follows from Theorem \ref{pdegeneration2} that  there is a finite set $S \subset K^{*}$ such that if $\mu^2 \not\in S$, then
 $\mu^2 f_{a_1,a_2;\mu u_1,\mu^{-1}u_2}(x)$  has no multiple roots, and therefore   $f_{0,a;\mu u_1,\mu^{-1}u_2}(x)$
 also has no multiple roots. It follows that $f_{0,a;\mu u_1,\mu^{-1}u_2}(x)$ has no multiple roots for all but finitely many $\mu\in K^{*}$.
 \end{proof}

 \section{Rationality Questions}
 \label{nonex}
 The aim of this section is to discuss the cases in which there are at most two $K_0$-rational points of order $2g+1$ on every odd degree genus $g$ hyperelliptic curve.

 \begin{thm}
 \label{realOddg}

 Let $K_0$ be a subfield of $K$ and  $g \ge 1$ be an integer.  Let us assume that $2g+1$ is not divisible by $\fchar(K)$ and the degree $2g$ monic polynomial
 $$\frac{x^{2g+1}-1}{x-1}=\sum_{i=0}^{2g} x^i \in K_0[x]$$
 does not have a factor in $K_0[x]$ of degree $g$ or equivalently cannot be represented as a product of two degree $g$ polynomials with coefficients in $K_0[x]$.

 Let $f(x)\in K_0[x]$ be a monic degree $2g+1$ polynomial without multiple roots and $\mathcal{C}_f: y^2=f(x)$ be the corresponding odd degree genus $g$ hyperelliptic curve defined over $K_0$.
 Then $\mathcal{C}_f(K_0)$ contains  at most  two torsion points of order $2g+1$.
 \end{thm}

\begin{proof}
 Assume that $\mathcal{C}_f(K_0)$ contains at least three points of order $2g+1$. Let $P \in \mathcal{C}_f(K_0)$ be one of them. Then $P=(a_1,c_1)$ with
 $$a_1,c_1 \in K_0, c_1 \ne 0,  \ c_1^2=f(a_1).$$
 The point $\iota(P)=(a_1,-c_1)\in \mathcal{C}_f(K_0)$  also has order $2g+1$. Hence there exists another point $Q \in \mathcal{C}_f(K_0)$ of order $2g+1$ that is neither $P$ nor $\iota(P)$. This implies that $Q=(a_2,c_2)$ with
 $$a_2,c_2 \in K_0, c_2 \ne 0,  \ c_2^2=f(a_2), \ a_2 \ne a_1.$$
 In particular, $\mathcal{C}_f(K_0)$ has four distinct points of order $2g+1$,
 \beq
 \label{for2gp1}
 P=(a_1,c_1),\iota(P)=(a_1,-c_1) ,Q=(a_2,c_2), \iota(Q)=(a_2,-c_2)\in \mathcal{C}_f(K_0).
 \eeq
 By Theorem \ref{PordK02g+1} applied to the torsion $K_0$-points $P=(a_1,c_1)$ and $Q=(a_2,c_2)$ of order $2g+1$,   there exist degree $g$ polynomials $u_1(x),u_2(x)\in K_0[x]$ such that
 $$\deg(u_1)=\deg(u_2)=g, \ u_1(x)u_2(x)=(x-a_2)^{2g+1}-(x-a_1)^{2g+1},$$
 $$u_1(a_1)\ne 0, u_2(a_1)\ne 0,  \ u_1(a_2)\ne 0,   u_2(a_2)\ne 0.$$
 This implies that
 \beq
 \label{xa1a2}
 (x-a)^{2g+1}-x^{2g+1}=u_1(x+a_1)u_2(x+a_1)=\tilde{u}_1(x)\tilde{u}_2(x),
 \eeq
 where
 $$a=a_2-a_1 \in K^{*}, \tilde{u}_1(x):=u_1(x+a_1), \tilde{u}_2(x):=u_2(x+a_1).$$
 Clearly, both $\tilde{u}_1(x)$ and $\tilde{u}_2(x)$ are polynomials of degree $g$ with coefficients in $K_0$,
 and their constant terms $\tilde{u}_1(0)=u_1(a_1)$ and   $\tilde{u}_2(0)=u_2(a_1)$ do not vanish. It follows from \eqref{xa1a2} that
 $$\tilde{u}_1(x)\tilde{u}_2(x)=(x-a)^{2g+1}-x^{2g+1}=(-a) \frac{(x-a)^{2g+1}-x^{2g+1}}{x\cdot(-a/x)}.$$
 On the other hand, dividing both sides of the latter equality by $x^{2g}$, we get
 $$\frac{\tilde{u}_1(x)}{x^g} \frac{\tilde{u}_2(x)}{x^g}=
 (-a)
 \frac{(x-a)^{2g+1}-x^{2g+1}}{x^{2g+1}((-a/x)}= (-a)\frac{(1-a/x)^{2g+1}-1}{(-a/x)}.$$
 Since both $\tilde{u}_1(x)$ and $\tilde{u}_2(x)$ are degree $g$ polynomials in $K_0[x]$ with nonzero constant terms, it follows from Lemma \ref{l2}
  that
 there  exist degree $g$ polynomials $w_1(x)$ and $w_2(x)$ in $K_0[x]$ such that
  $$\frac{\tilde{u}_1(x)}{x^g}=w_1(-a/x),  \ \frac{\tilde{u}_1(x)}{x^g}=w_1(-a/x).$$
  This implies that
  $$w_1(-a/x)w_2(-a/x)=(-a)\frac{(1-a/x)^{2g+1}-1}{-a/x}.$$
Hence
  $$w_1(x)w_2(x)=(-a)\frac{(x+1)^{2g+1}-1}{x},$$
 and therefore
 $$\frac{(x+1)^{2g+1}-1}{x}=\frac{w_1(x)}{-a}w_2(x).$$
 It follows that the polynomial
 $$\frac{x^{2g+1}-1}{x-1}=\frac{w_1(x-1)}{-a}w_2(x-1)$$
 splits into a product of two degree $g$ polynomials $w_1(x-1)/(-a)$ and $w_2(x-1)$ with coefficients in $K_0$,
 which contradicts our assumptions. The obtained contradiction proves the desired result.
 \end{proof}

\begin{example}
Suppose that $g=1$ and $\fchar(K)\ne 3$. Assume that
$$\frac{x^3-1}{x-1}=x^2+x+1$$
does not split into a product of linear factors,  i.e.,  $K_0$ does not contain a primitive cubic root of unity.  On the other hand, $f(x)$ is a cubic polynomial and $\mathcal{C}_f$ is an elliptic curve. It follows from Theorem \ref{realOddg} that  $\mathcal{C}_f(K_0)$ contains  at most  two points of order $3$ (which is well known). In this case one may give a direct proof.
Namely, suppose  $\mathcal{C}_f(K_0)$ contains at least three points of order $3$, then one may find two of them say, $P, Q \in \mathcal{C}_f(K_0)$ such that $Q \ne P, \iota(P)=-P$, and therefore the value of the corresponding Weil pairing $e_3(P,Q)$ between them is a primitive cubic root of unity.
Since both $P$ and $Q$ lie in $\mathcal{C}_f(K_0)$, the root
$e_3(P,Q)$ lies in $K_0$, which contradicts our assumptions.
\end{example}

\begin{cor}
\label{realNO}
Suppose that $K$ is the field $\mathbb{C}$ of complex numbers and $K_0$ is its subfield $\mathbb{R}$ of real numbers.  Suppose that $g$ is a positive odd integer, $f(x)\in \mathbb{R}[x]$ is a monic degree $2g+1$ polynomial  without multiple roots, and $\mathcal{C}_f: y^2=f(x)$ is the corresponding odd degree genus $g$ hyperelliptic curve over  $\mathbb{R}$.  Then $\mathcal{C}_f(\mathbb{R})$ contains at most two points of order $2g+1$.
\end{cor}

\begin{proof}
Notice that the polynomial $(x^{2g+1}-1)/(x-1)$ has no real roots, because $2g+1$ is odd. Suppose that it splits into a product
$$\frac{ x^{2g+1}-1 }{ x-1 }=w_1(x)w_2(x)$$
 of two real polynomials $w_1(x)$ and $w_2(x)$, both of degree $g$. Since $g$ is odd, both $w_1(x)$ and $w_2(x)$ have a real root, and therefore  $(x^{2g+1}-1)/(x-1)$ also has  a real root, which is not true. So, $(x^{2g+1}-1)/(x-1)$ does not split into a product of two real polynomials of degree $g$. Now the desired result follows from
Theorem \ref{realOddg}.
\end{proof}

\begin{thm}
\label{infinite}
Let $K_0$ be an infinite  subfield of $K$ and  $g \ge 1$ be an integer.  Let us assume that $2g+1$ is not divisible by $\fchar(K)$.
Then the following conditions are equivalent.
\begin{itemize}
\item[(i)]
The degree $2g$ monic polynomial
 $$\frac{x^{2g+1}-1}{x-1}=\sum_{i=0}^{2g} x^i \in K_0[x]$$
 has a factor in $K_0[x]$ of degree $g$ or equivalently can be represented as a product of two degree $g$ polynomials with coefficients in $K_0$.
 \item[(ii)]
There exists  a monic degree $2g+1$ polynomial $f(x)\in K_0[x]$ without multiple roots  that enjoys the following property.
If $\mathcal{C}_f: y^2=f(x)$ is  the corresponding odd degree genus $g$ hyperelliptic curve  defined over $K_0$,
 then $\mathcal{C}_f(K_0)$ contains at least four torsion points of order $2g+1$.
 \end{itemize}
\end{thm}

\begin{proof}
The implication (ii) $\Longrightarrow$(i) follows from Theorem \ref{realOddg}  and its proof.

Suppose (i) holds, i.e.,  there exist two  degree $g$ polynomials $w_1(x), w_2(x)\in K_0[x]$
such that
$$w_1(x)w_2(x)=\frac{x^{2g+1}-1}{x-1}=\sum_{i=0}^{2g} x^i.$$
In particular,
$$w_1(1) w_2(1)=2g+1 \ne 0,$$
and therefore $w_1(1)\ne 0, w_2(1)\ne 0$.
This means that
$$\tilde{w}_1(x)\tilde{w}_2(x)=\frac{(x+1)^{2g+1}-1}{x},$$
where
$$\begin{aligned}\tilde{w}_1(x)&=w_1(x+1)\in K_0[x], \ \tilde{w}_2(x)=w_2(x+1)\in K_0[x], \\ \tilde{w}_1(0)&=w_1(1)\ne 0,\ \tilde{w}_2(0)=w_2(1)\ne 0.\end{aligned}$$
Clearly, both $\tilde{w}_1(x), \tilde{w}_2(x)$ are degree $g$ polynomials with nonzero constant terms. We have
\beq\label{oneOverx}
(1+1/x)^{2g+1}- (1/x)^{2g+1}=\frac{(x+1)^{2g+1}-1}{x^{2g+1}}=\frac{\tilde{w}_1(x)}{x^g} \frac{\tilde{w}_2(x)}{x^g}.
\eeq
By Lemma \ref{l2}, there  exist degree $g$ polynomials $u_1(x), u_2(x)\in K_0[x]$ such that
$$u_1(1/x)=\frac{\tilde{w}_1(x)}{x^g}, \ u_2(1/x)=\frac{\tilde{w}_2(x)}{x^g}.$$
It follows from \eqref{oneOverx} that
$$(1+1/x)^{2g+1}- (1/x)^{2g+1}=u_1(1/x)u_2(1/x),$$
and therefore
$$(x+1)^{2g+1}-x^{2g+1}=u_1(x)u_2(x).$$
Since $K_0$ is infinite, it follows from Theorem \ref{degeneration} that there exists $\mu\in K_0^{*}$ such that the polynomial
$$f_{0, -1;\mu u_1,\mu^{-1}u_2}(x)=
 x^{2g+1}+\left(\frac{\mu u_1(x)+\mu^{-1}u_2(x)}{2}\right)^2$$
 has no multiple roots. By Theorem \ref{ordKf2g+1},   the odd degree genus $g$ hyperelliptic curve
 $$\mathcal{C}_{0,-1;\mu u_1\mu^{-1}u_2}: y^2=f_{0,-1;\mu u_1,\mu^{-1}u_2}(x)$$
 over $K_0$ has  two distinct  points
 $$P_{0,-1;\mu u_1,\mu^{-1}u_2}, Q_{0,-1;\mu u_1,\mu^{-1}u_2}\in \mathcal{C}_{0,-1;\mu u_1,\mu^{-1}u_2}(K_0)$$   of order $2g+1$
 with abscissas $0$ and $-1$, respectively, and with nonzero ordinates. Consequently,
 $$P_{0,-1;\mu u_1,\mu^{-1}u_2}, \ Q_{0,-1;\mu u_1,\mu^{-1}u_2}, \ \iota(P_{0,-1;\mu u_1\mu^{-1},u_2}),\
 \iota(Q_{0,-1;\mu u_1,\mu^{-1}u_2}) $$
 are four distinct $K_0$-rational points of order $2g+1$ on $ \mathcal{C}_{0,-1;\mu u_1,\mu^{-1}u_2}$.  This  implies that (ii) holds.
\end{proof}

\begin{example}
\label{realYES}
Suppose that $g=2d$ is an {\sl even} positive integer, $K=\mathbb{C}, K_0=\mathbb{R}$.  Then the $2g$-element set of complex roots of $(x^{2g+1}-1)/(x-1)$ consists of $g$ complex-conjugate pairs $\{\zeta_1,\bar{\zeta}_1; \dots ; \zeta_{g}, \bar{\zeta}_g\}$. Let us consider the degree $g$ polynomials
$$w_1(x)=\prod_{i=1}^d (x-\zeta_i)(x-\bar{\zeta}_i), \ w_2(x)=\prod_{i=d+1}^g (x-\zeta_i)(x-\bar{\zeta}_i).$$
Clearly,
$$w_1(x), w_2(x)\in \mathbb{R}[x], \ w_1(x)w_2(x)=\frac{x^{2g+1}-1}{x-1}.$$
It follows from Theorem \ref{infinite}  that
there exists  a monic degree $2g+1$ polynomial $f(x)\in \mathbb{R}[x]$ without multiple roots  such that the  odd degree genus $g$ hyperelliptic curve
 $\mathcal{C}_f: y^2=f(x)$
  contains at least four {\sl real} torsion points of order $2g+1$.
\end{example}

Theorem \ref{infinite}  suggests the following definition.

\begin{defn}
Let $\varphi(n)$ be Euler's totient function.
An odd integer $2g+1\ge 3$ is called {\sl hyperelliptic} if it enjoys the following obviously equivalent properties.

\begin{itemize}
\item[(i)]
There is a set $S$ of divisors  of $2g+1$ that does {\sl not} contain $1$ and such that
$$\sum_{d\in S}\varphi(d)=g.$$
\item[(ii)]
One may partition the set  of all divisors of   $2g+1$ except $1$ into two nonempty subsets $S_1$ and $S_2$ such that
$$\sum_{d\in S_1}\varphi(d)=\sum_{d\in S_2}\varphi(d).$$
\end{itemize}
\end{defn}

\begin{thm}
\label{hyperNumber}
Suppose that $K$ is the field $\mathbb{C}$ of complex numbers and $K_0$ is its subfield $\mathbb{Q}$  of rational numbers.  Suppose that $g$ is a positive odd integer.
Then the following conditions are equivalent.
\begin{itemize}
\item[(i)]
$2g+1$ is a hyperelliptic number.
\item[(ii)]
There exists a monic degree $2g+1$ polynomial
 $f(x)\in \mathbb{Q}[x]$  with rational coefficients and without multiple roots  that  enjoys the following property. If $\mathcal{C}_f: y^2=f(x)$ is
  the corresponding odd degree genus $g$ hyperelliptic curve
   defined over  $\mathbb{Q}$,  then  $\mathcal{C}_f(\mathbb{Q})$ contains  at least   four torsion points of order $2g+1$.
\end{itemize}
\end{thm}

\begin{proof}
Let $D(2g+1)$ be the set of all divisors of $2g+1$ except $1$.  Then the monic
 polynomial
$(x^{2g+1}-1)/(x-1)$ coincides with the product $\prod_{d\in D(2g+1)}\Phi_d(x)$
of distinct cyclotomic polynomials $\Phi_d(x)$, each of which is  irreducible over $\mathbb{Q}$.
This implies that each  factor $w(x)$ of $(x^{2g+1}-1)/(x-1)$ in $\mathbb{Q}[x]$ is of the form $r\cdot \prod_{d\in S}\Phi_d(x)$, where $S$ is a subset in $D(2g+1)$ and $r\in \mathbb{Q}^{*}$. Since $\deg(\Phi_d)=\varphi(d)$,
we have $$\deg(w)=\sum_{d\in S}\varphi(d).$$
The desired result follows readily from Theorem \ref{infinite} applied to $K_0=\mathbb{Q}$.
\end{proof}

\begin{example}
\label{EXoverQ}
Let  $K_0=\mathbb{Q}, K=\mathbb{C}$.
\begin{itemize}
\item[(i)]
 Let us take $g=52$. Then
$2g+1=105=3 \cdot 5 \cdot 7,$
$$\varphi(105)=48,\ \varphi(5)=4, \ 52=48+4=\varphi(105)+\varphi(5).$$
Hence $105$ is a hyperelliptic number and there exists a degree $105$ polynomial $f(x)\in \mathbb{Q}[x]$ without multiple roots such that
the corresponding odd degree genus $52$ hyperelliptic $\mathbb{Q}$-curve $\mathcal{C}_f: y^2=f(x)$ has at least   four $\mathbb{Q}$-points
of order $105$.
\item[(ii)]
 Let us take $g=82$. Then
$2g+1=165=3 \cdot 5 \cdot 11,$
$$\varphi(165)=80, \varphi(3)=2, \ 82=80+2=\varphi(165)+\varphi(3).$$
This implies that $165$ is a hyperelliptic number and
 there exists a degree $165$ polynomial $f(x)\in \mathbb{Q}[x]$ without multiple roots such that
the corresponding odd degree genus $82$ hyperelliptic $\mathbb{Q}$-curve $\mathcal{C}_f: y^2=f(x)$ has  at least   four $\mathbb{Q}$-points
of order $165$.
\end{itemize}
\end{example}

\begin{cor}\label{overQ}
Suppose that $K$ is the field $\mathbb{C}$ of complex numbers and $K_0$ is its subfield $\mathbb{Q}$  of rational numbers.  Suppose that $g$ is a positive  integer enjoying one of the following properties.

\begin{itemize}
\item[(i)]
There exist a prime $\ell$ and a positive integer $k$ such that $2g+1=\ell^k$.
\item[(ii)]
There exist distinct odd primes $\ell_1$ and  $\ell_2$, and  positive integers $k_1$ and $k_2$ such that $2g+1=\ell_1^{k_1}\ell_2^{k_2}$.
\item[(iii)]
There exist distinct odd primes $\ell_1$, $\ell_2$, $\ell_3$ and  positive integers $k_1, k_2, k_3$ such that $2g+1=\ell_1^{k_1}\ell_2^{k_2}\ell_3^{k_3}$
and none of $\ell_i$ is $3$.
\item[(iv)] $g \le 100$ and $g \not\in\{52,82\}$.
\end{itemize}

Then:
\begin{itemize}
\item[(i)]
$2g+1$ is not a hyperelliptic number.
\item[(ii)]
Let $f(x)\in \mathbb{Q}[x]$ be a monic degree $2g+1$ polynomial    without multiple roots  and $\mathcal{C}_f: y^2=f(x)$ the corresponding odd degree genus $g$ hyperelliptic curve   defined over  $\mathbb{Q}$. Then
 $\mathcal{C}_f(\mathbb{Q})$ contains  at most  two points of order $2g+1$.
 \end{itemize}
\end{cor}

\begin{proof}
In light of Theorem \ref{hyperNumber}, it suffices to check that $2g+1$ is {\sl not} a hyperelliptic number.  Let us assume the contrary, i.e., that
one may partition $D(2g+1)$ into two  subsets $S_1$ and $S_2$ such that
$$\sum_{d\in S_1}\varphi(d)
=g=\sum_{d\in S_2}\varphi(d).$$

  Case (i). We have $\ell\ge 3$ and
   $$\varphi(2g+1)=
   (\ell-1)\ell^{k-1}\ge \frac{2}{3}\ell^k> \frac{2}{3}{2g}=\frac{4}{3}g>g.$$

   Case (ii).  We may assume that $\ell_2>\ell_1$, and therefore
   $\ell_1 \ge 3, \ell_2 \ge 5$. We have
   $$\varphi(2g+1)=(\ell_1-1)\ell_1^{k-1}(\ell_2-1)\ell_2^{k_2-1}\ge $$
  $$ \frac{2}{3}\ell_1^{k_1}\cdot \frac{4}{5}\ell_2^{k_2}=\frac{8}{15}(\ell_1^{k_1}\cdot \ell_2^{k_2})=\frac{8}{15}(2g+1)>\frac{16}{15}g>g.$$

  Case (iii).  We may assume that $\ell_3>\ell_2>\ell_1>3$, and therefore
  $$\ell_1 \ge 5, \ell_2 \ge 7, \ell_3 \ge 11.$$
   We have
  $$\varphi(2g+1)=(\ell_1-1)\ell_1^{k-1}(\ell_2-1)\ell_2^{k_2-1}(\ell_3-1)\ell_3^{k_3-1}\ge $$
  $$ \frac{4}{5}\ell_1^{k_1}\cdot \frac{6}{7}\ell_2^{k_2}\cdot \frac{10}{11}\ell_3^{k_3}=\frac{48}{77}(\ell_1^{k_1} \ell_2^{k_2}\ell_3^{k_3})=\frac{48}{77}(2g+1)>\frac{96}{77}g>g.$$
In all three cases $\varphi(2g+1)>g$.  Since $2g+1 \in S_i$ for $i=1$ or $2$,
$$g=\sum_{d\in S_i}\varphi(d) \ge \varphi(2g+1)>g,$$
which gives us a desired contradiction.

Let us assume that  case (iv) holds.
It follows from Corollary \ref{realNO} that we may assume that $g$ is even.
We may also assume that $g$ satisfies neither (i) nor (ii).
Since $g$ satisfies neither (i) nor (ii),   $2g+1$ is divisible by at least three distinct odd primes,
hence $2g+1 \ge 3 \cdot 5 \cdot 7=105$, i.e., $g >51$.  So, we may assume that $52<g\le 100$.

If $2g+1$ is not divisible by $3$, then
 $2g+1\ge  5\cdot 7\cdot 11=385$, i.e., $g>191>118$. Hence $2g+1$ is divisible by $3$. Since $g$ is even, it is congruent to $4$ modulo $6$.
This implies that $g \in \{58,64,70,76,88, 94,100\}$. However,
$$2\cdot 58+1=3^2\cdot 13,\ 2\cdot 64+1=3 \cdot 43,\ 2\cdot 70+1=3\cdot 47,\ 2\cdot 76+1=3^2\cdot 17,$$
$$2\cdot 88+1=3\cdot 59,\ 2\cdot 94+1=3^3\cdot 7,\ 2\cdot 100+1=3\cdot 67.$$
Consequently, every $g \in \{58,64,70,76, 88, 94, 100\}$ satisfies (ii). This ends the proof.
\end{proof}

\begin{Rem} Our results show that there are only two hyperelliptic numbers $2g+1 \le 201$,
namely, $105$ and $165$.
However, recently Vlad Matei \cite{Matei} proved that the set of hyperelliptic numbers is infinite.
\end{Rem}

The following assertion may be viewed as a counterpart in characteristic zero to Theorem \ref{char2gPlus1}.

\begin{thm}\label{Ladic}
Let $\ell$ be an odd prime and  $K_0$ a complete discrete valuation field of characteristic $0$ with residue field of characteristic
$\ell$ and such that  the ramification index $e_K$ is $1$, i.e., $\ell$ is a uniformizer. (E.g., $K_0$ is the
field $\mathbb{Q}_{\ell}$ of $\ell$-adic numbers or its finite unramified extension). Let $K$ be an algebraic closure of $K_0$.
Suppose that there exists a positive integer $k$ such that $g=(\ell^k-1)/2$, i.e., $2g+1=\ell^k$.

Let $f(x)\in K_0[x]$ be a monic degree $\ell^k$ polynomial without multiple roots and
$\mathbb{C}_f: y^2=f(x)$ the corresponding odd degree genus $(\ell^k-1)/2$ hyperelliptic curve over $K_0$.
Then   $\mathbb{C}_f(K_0)$ has at most  two points of order $\ell^k$.
\end{thm}

\section{Odd degree genus $g$ hyperelliptic curves with two pairs of torsion points of order $2g+1$.}\label{sec3}
\label{closedK}

In this section we assume that $K$ is an algebraically closed field of characteristic $\neq2$.
We will need the following definition.

\begin{defn}
\label{nice}
Let $g$ be a positive integer.
An ordered pair of polynomials
$$u_1(x), u_2(x)\in K[x]$$
is called a {\sl nice pair} of degree $g$ over $K$ if it enjoys the following properties.

\begin{itemize}
\item[(i)]  $\deg(u_1)\le g, \ \deg(u_2)\le g$.
\item[(ii)]  $u_1(x)u_2(x)=(x+1)^{2g+1}-x^{2g}$.
\item[(iii)] If $\fchar(K)$ does {\sl not} divide $2g+1$, then
 $\deg(u_1)= g,  \deg(u_2)=g.$
 \item[(iv)]
 $u_1^{\prime}(x)\ne 0, \ u_2^{\prime}(x)\ne 0.$
\end{itemize}
If $(u_1(x), u_2(x))$ is a nice pair of degree $g$ and the polynomial
$$\begin{aligned}f(x)=f_{0,-1; u_1,u_2}=x^{2g+1}+\left(\frac{u_1(x)+u_2(x)}{2}\right)^2\\=(x+1)^{2g+1}+\left(\frac{u_1(x)-u_2(x)}{2}\right)^2\end{aligned}$$
has {\sl no} multiple roots, then the pair $(u_1(x), u_2(x))$ is called {\sl very nice}.
\end{defn}

\begin{Rem}
\label{niceR}
Suppose that  $(u_1(x), u_2(x))$ is a nice pair of degree $g$.
\begin{itemize}
\item[(i)]
It follows from Remark \ref{produ1u2} that
$$u_1(0)\ne 0, u_2(0) \ne 0, \ u_2(-1)\ne 0, u_2(-1)\ne 0.$$
 In particular,
 $$u_2(x) \ne \pm u_1(x).$$
In addition, if $(u_1(x), u_2(x))$ is very nice, then it follows from Theorem \ref{ordKf2g+1}
that
$$u_1(0)+u_2(0) \ne 0, \ u_2(-1)-u_2(-1)\ne 0.$$
\item[(ii)]
Obviously, the pairs $$(-u_1(x), -u_2(x)),\ (u_2(x), u_1(x)),\  (-u_2(x), -u_1(x))$$
are also nice of degree $g$. It follows from (i) that all four nice pairs (including $(u_1(x),u_2(x))$
are {\sl distinct}.
However, they all give rise to the same polynomial $f(x)$ (see Remark \ref{changesign}).
In particular, they all are very nice if and only if $(u_1(x),u_2(x))$ is very nice.
\item[(iii)]
If $\mu \in K^{*}$, then obviously $(\mu u_1(x), \mu^{-1}u_2(x))$ is a nice pair of degree $g$.
It follows from Theorems \ref{pdegeneration2}   and  \ref{degenerationCHARp} that
$(\mu u_1(x), \mu^{-1}u_2(x))$ is actually very nice for all but finitely many $\mu$.
\item[(iv)]
Let $(w_1(x),w_2(x))$ be a nice pair of degree $g$ such that
$$f_{0,-1; w_1,w_2}(x)=f_{0,-1; u_1,u_2}(x).$$
Then $(w_1(x),w_2(x))$ is one of four pairs described in (ii). Indeed, we immediately get
$$\begin{aligned}\left(\frac{w_1(x)+w_2(x)}{2}\right)^2=\left(\frac{u_1(x)+u_2(x)}{2}\right)^2, \\
\left(\frac{w_1(x)+w_2(x)}{2}\right)^2=\left(\frac{u_1(x)+u_2(x)}{2}\right)^2.\end{aligned}$$
It follows that we have  at most  four choices for $$(w_1(x)+w_2(x), w_1(x)-w_2(x)),$$
and therefore  at most  four choices for  $(w_1(x),w_2(x))$. However, in (ii) we already
described the four choices,  and therefore $(w_1(x),w_2(x))$ is one of them.
\end{itemize}
\end{Rem}

\begin{defn}
A monic degree $2g+1$ polynomial $f(x)\in K[x]$ is called {\sl decorated}
if there exists a nice pair $(u_1(x),u_2(x))$ of degree $g$ such that
$f(x)=f_{0,-1; u_1,u_2}(x)$. If this is the case, then $(u_1(x),u_2(x))$ is called a {\sl decoration}
of $f(x)$. It follows from Remark \ref{niceR} that a decorated polynomial admits precisely
four decorations.
\end{defn}

These definitions allow us to restate results of Section \ref{twoPoints} in the following way.

\begin{thm}
\label{restate}
Let $f(x)$ be a monic polynomial of degree $2g+1$ without multiple roots, and let
$\mathcal{C}_f:y^2=f(x)$ be the corresponding odd degree genus $g$ hyperelliptic curve over $K$.
\begin{itemize}
\item[(i)]
Let $P$ and $Q$ be points in $\mathcal{C}_f(K)$ such that
$$x(P)=0,\ x(Q)=-1.$$
Then both $P$ and $Q$ have order $2g+1$ if and only if $f(x)$ is decorated.
\item[(ii)]
Suppose that $f(x)$ is decorated. Then each decoration $(u_1(x),u_2(x))$ of $f(x)$
gives rise to points
\begin{equation}
\label{PQu}
P_{u_1,u_2}:=\left(0, \frac{u_1(0)+u_2(0)}{2}\right), \ Q_{u_1,u_2}:=\left(-1,\frac{u_1(-1)-u_2(-1)}{2}\right) \in \mathcal{C}_f(K)
\end{equation}
of order $2g+1$.

Conversely, for each pair of points $P, Q \in \mathcal{C}_f(K)$  with
$$x(P)=0,\ x(Q)=-1$$
there exists exactly one decoration
 $(u_1(x),u_2(x))$ of $f(x)$ such that
 \begin{equation}
 \label{PQu1u2}
 P=\left(0, \frac{u_1(0)+u_2(0)}{2}\right),\ Q=\left(-1, \frac{u_1(-1)-u_2(-1)}{2}\right).
 \end{equation}
 In addition, both $P$ and $Q$ have order $2g+1$.
 \end{itemize}
\end{thm}

\begin{proof}
(i) Suppose  $P$ and $Q$ have order $2g+1$. It follows from Theorem \ref{PordK02g+1}
and Theorem \ref{ordKf2g+1}(c1) applied to $a_1=0, a_2=-1$ that $f(x)$ is decorated.
Conversely, suppose $f(x)$ is decorated. It follows from Theorem  \ref{ordKf2g+1}(c1) applied to $a_1=0, a_2=-1$
that there exist torsion points
$P_1,Q_1\in \mathcal{C}_f(K)$ of order $2g+1$ such that
$$x(P_1)=0, \ x(Q_1)=-1.$$
This implies that $P=P_1$ or $P=\iota(P_1)$ and $Q=Q_1$ or $Q=\iota(Q_1)$. In all the cases,
$P$ and $P_1$ have the same order,  $Q$ and $Q_1$ have the same order. This implies that both $P$ and $Q$ have order $2g+1$.

(ii) Suppose that $f(x)$ is decorated.

Let $(u_1(x),u_2(x))$ be a decoration of $f(x)$. It follows from Theorem  \ref{ordKf2g+1}(c1) applied to $a_1=0, a_2=-1$
that $P_{u_1,u_2}$ and $Q_{u_1,u_2}$ are indeed torsion points in $\mathcal{C}_f(K)$ and have order $2g+1$.

Let $P, Q \in \mathcal{C}_f(K)$  and
$x(P)=0, x(Q)=-1$. It follows from (i) that both $P$ and $Q$ have  order $2g+1$.
Now it follows from Theorem \ref{PordK02g+1}
and Theorem \ref{ordKf2g+1}(c1) applied to $a_1=0, a_2=-1$
that there is precisely one decoration $(u_1(x),u_2(x))$ of $f(x)$ such that $P$ and $Q$ are defined
by \eqref{PQu1u2}.
\end{proof}

\begin{defn}
\begin{itemize}
\item[(i)]
An {\textsl{enhanced  hyperelliptic curve}} of genus $g$ over $K$ is an ordered
quadruple $(\mathcal C, O, P, Q)$, where $(\mathcal C, O)$ is a pointed odd degree  genus $g$ hyperelliptic curve  and  $P,Q$ are  points of order $2g+1$ such that $Q\neq P,  \iota P$.

We call an enhanced hyperelliptic curve  of genus $g$ over $K$ {\sl normalized} if there exists a polynomial $f(x)\in K[x]$ of degree $2g+1$ without multiple roots such that
$\mathcal{C}=\mathcal{C}_f$, i.e., if $\mathcal{C}$ is the smooth projective model of $y^2=f(x)$, $O=\infty$, and
$x(P)=0,\ x(Q)=-1$.

\item[(ii)]
  By an \textsl{isomorphism} $\phi: (\mathcal C, O, P, Q )\to  ( \mathcal C_1, O_1,  P_1, Q_1 )$ of enhanced  hyperelliptic curves we mean a $K$-biregular map $\phi:\mathcal C\to \mathcal C_1$ such that $\phi(O)=O_1$,
$\phi(P)=P_1$, and $\phi(Q)=Q_1$.
We call  an  isomorphism  $\phi: (\mathcal C, O, P, Q )\to  ( \mathcal C_1, O_1,  P_1, Q_1 )$ of enhanced  hyperelliptic curves a {\sl marking}
if $\mathcal{C}_1=\mathcal{C}_{f_1}$ is the smooth projective model of $y_1^2=f(x_1)$, where $f(x_1)\in K[x_1]$ is a degree $2g+1$ polynomial without multiple roots,
  $O_1$ the infinite point $\infty_1$ of $\mathcal{C}_{f_1}$ and
  $x_1(P_1)=0, \ x_1(Q_1)=-1$.
  In other words,  a marking of $(\mathcal C, O, P, Q )$ is an isomorphism between $(\mathcal C, O, P, Q )$ and a normalized enhanced hyperelliptic curve.
\end{itemize}
\end{defn}

\begin{Rem}
\begin{itemize}
\item[(i)]
Notice that if $\phi: (\mathcal C,O)\to (\mathcal C_1,O_1)$ is a $K$-biregular map of pointed hyperelliptic curves and $P$ is a $K$-point of $\mathcal C$ having order $2g+1$ on the Jacobian
$J(\mathcal C)$ of $\mathcal C$, then the $K$-point $\phi(P)$
of $\mathcal C_1$ has order $2g+1$ on the Jacobian $J(\mathcal C_1)$ of $\mathcal C_1$.
Consequently, every $K$-biregular map $\phi: (\mathcal C,O)\to (\mathcal C_1,O_1)$ of pointed hyperelliptic curves yields an isomorphism $\phi: (\mathcal C, O, P, Q )\to  ( \mathcal C_1, O_1,  P_1, Q_1 )$ of enhanced  hyperelliptic curves, where $P$ and $Q$ are arbitrary points of order $2g+1$ on $C$ and $P_1=\phi(P)$, $Q_1=\phi(Q) $.
\item[(ii)]
Recall (Section \ref{prel}) that every pointed genus $g$ hyperelliptic curve $(\mathcal C,O)$ is $K$-isomorphic  to $(\mathcal C_f,\infty)$, where $\mathcal C_f$ is the odd degree genus $g$ hyperelliptic curve defined by the equation $y^2=f(x)$ (i.e., the normalization of the projective closure of the smooth plane affine curve
 $y^2=f(x)$) and $\infty$ is the unique ``infinite'' point on $C_f$. Therefore, every  enhanced  hyperelliptic curve is isomorphic to a  enhanced  hyperelliptic curve  $(\mathcal C_f,\infty, P,Q)$.
 \end{itemize}
 \end{Rem}

\begin{thm}\label{0,-1}
Let $(\mathcal C,O, P,Q)$ be an enhanced  genus $g$  hyperelliptic curve, where   $\mathcal C_f$ is the odd degree genus $g$ hyperelliptic curve defined by the equation $y^2=f(x)$.
Then there exists a  degree $2g+1$ monic polynomial $f_1(x)\in K[x]$ without multiple roots and an  enhanced   genus $g$ hyperelliptic curve $(\mathcal C_{f_1},\infty, P_1,Q_1)$
 that enjoys the following properties.

 \begin{itemize}
\item[(i)] $x(P_1)=0$ and $x(Q_1)=-1$, i.e., $(\mathcal C_{f_1},\infty, P_1,Q_1)$ is normalized.
\item[(ii)] The enhanced hyperelliptic curves  $(\mathcal C_{f_1},\infty, P_1,Q_1)$ and $(\mathcal C,O, P,Q)$ are isomorphic.
\end{itemize}
In other words, every enhanced  genus $g$  hyperelliptic curve admits a marking.
\end{thm}
\begin{proof}
Without loss of generality we may assume that
$$\mathcal{C}=\mathcal{C}_f: y^2=f(x),$$
where $f(x)\in K[x]$ is a degree $2g+1$ monic polynomial without multiple roots.
Let
$$P=(a,b)\in \mathcal{C}_f(K), \ Q=(c,d) \in \mathcal{C}_f(K).$$
 Then $a$ and $c$ are {\sl distinct} elements of $K$, none of which is a root of $f(x)$, i.e.,
 $b \ne 0,\ d \ne 0.$
Let us consider the monic degree $2g + 1$ polynomial
 $$f_1(x)=\frac{f((a-c)x+a)}{(a-c)^{2g+1}}\in K[x]$$
without multiple roots and the hyperelliptic curve $\mathcal C_1$ defined by the equation
$y^2=f_1(x)$. Let us choose
$$r=\sqrt{a-c}\in K^{*}.$$
 Then we get a $K$-isomorphism of pointed hyperelliptic curves
 $$\phi:(\mathcal C_f,\infty)  \to (\mathcal{C}_{f_1},\infty), \ \phi(x,y)=\left(\frac{x-a}{a-c}, \ r (a-c)^gy\right),$$
 which gives rise to a $K$-isomorphism $$\phi: (\mathcal C_f,\infty, P,Q)\to (\mathcal C_{f_1},\infty, P_1,Q_1)$$
 of enhanced  hyperelliptic curves, where
 $P_1=\phi(P)=(0, r (a-c)^gb)$ and $Q_1=\phi(Q)=(-1, r (a-c)^gd)$.
 \end{proof}

\begin{Rem}
 Let $(\mathcal C_{f_1},\infty, P_1, Q_1)$ and  $(\mathcal C_{f_2},\infty, P_2,Q_2)$ be  two normalized  enhanced   hyperelliptic curves. In particular,
  the abscissas of both $P_1$ and $P_2$ equal $0$ and the abscissas of both $Q_1$ and $Q_2$ equal $-1$.
 \begin{itemize}
 \item[(i)]
 If there exists an isomorphism
 $$\psi: (\mathcal C_{f_1},\infty, P_1, Q_1) \cong (\mathcal C_{f_2},\infty, P_2,Q_2)$$  of  enhanced   hyperelliptic curves,
 then it follows from Remark~ \ref{isoPoint} that
 $$f_1(x)=f_2(x), \ \mathcal C_{f_1}=\mathcal C_{f_2},$$
and  $\psi$ is either the identity map or $\iota$. Consequently,
  either $P_2=P_1, Q_2= Q_1$ or $P_2=\iota P_1, Q_2=\iota Q_1$.
  This implies that every {\sl automorphism} $(\mathcal C_{f_1},\infty, P_1, Q_1)\cong (\mathcal C_{f_1},\infty, P_1, Q_1)$
  of a  normalized  enhanced   hyperelliptic curve is the {\sl identity map}.
 \item[(ii)]
 Let  $(\mathcal C,O, P, Q)$ be an  enhanced   genus $g$ hyperelliptic curve over $K$.
 Suppose that
 $$\phi_1: (\mathcal C,O, P, Q) \to  (\mathcal C_{f_1},\infty, P_1, Q_1), \  \phi_1: (\mathcal C,O, P, Q) \to  (\mathcal C_{f_2},\infty, P_2, Q_2)$$
 are two markings of  $(\mathcal C,O, P, Q)$. Then
 $$\psi:=\phi_2\circ \phi_1^{-1}:(\mathcal C_{f_1},\infty, P_1, Q_1) \to (\mathcal C_{f_2},\infty, P_2, Q_2)$$
 is an isomorphism of  enhanced   hyperelliptic curves that satisfies conditions (i). It follows that
 $$f_1(x)=f_2(x), \ \mathcal C_{f_1}=\mathcal C_{f_2},$$
 and either $\psi_2=\psi_1$ or $\psi_2=\psi_1\circ\iota_{\mathcal{C}}$.
 Therefore, every enhanced   hyperelliptic curve has exactly two markings, one is obtained from the other
 by composing with the hyperelliptic involution.
 \end{itemize}
 \end{Rem}

 \begin{Rem}
Let $(\mathcal C_f,\infty, P_2,Q_2)$  be a normalized  enhanced   hyperelliptic curve over $K$.
 By Theorem~ \ref{restate},
there exists precisely one decoration   $(u_1(x),u_2(x))$ of $f(x)$ such that
 \beq \label{PQ0m1}
 P=\left(0,\frac{u_1(0)+u_2(0)}{2}\right),
 \ Q=\left(-1,\frac{u_1(-1)-u_2(-1)}{2}\right).\eeq
It follows from Remarks \ref{changesign} and \ref{niceR} that  the same pointed  hyperelliptic curve  $(\mathcal C_f,\infty)$ gives rise to three other normalized enhanced
hyperelliptic curves
    $(\mathcal C_f,\infty, \iota P, \iota Q)$, $(\mathcal C_f,\infty,  P,\iota Q)$, $(\mathcal C_f,\infty, \iota P,  Q)$
that correspond to the very nice pairs
$$(-u_1(x), -u_2(x)), \ (u_2(x), u_1(x)), \ (-u_2(x), -u_1(x)),$$ respectively.
\end{Rem}


Now our goal is to describe the nice pairs $(u_1(x), u_2(x))$ explicitly. In what follows
we write $\#(A)$ for the cardinality of a finite set $A$.

\subsection{The case when $\fchar(K)$ does not divide  $2g+1$}
\label{charNotDiv}
Recall that in this case  each of the polynomials $u_1(x)$ and $u_2(x)$ has degree $g$.
Let us put  $$M(2g+1):=\{\eps \in K, \eps^{2g+1}=1, \eps\ne 1 \}.$$
The degree $2g$ polynomial $(x+1)^{2g+1}-x^{2g+1}$ has leading coefficient $2g+1$ and $2g$ distinct roots  $$\eta(\eps)= \frac1{\eps-1}, \; \text{where}\; \eps\in M(2g+1).$$
We write
$$\Eta_I(x)=\prod\limits_{\eps\in I}(x-\eta(\eps))\in K[x]$$
for each subset $I\subset M(2g+1)$. Clearly, $\Eta_I(x)$ is a degree $\#(I)$ monic polynomial;
$\Eta_I^{\prime}(x)=0$ if and only if $I=\emptyset$. It is also clear that if $\complement I$  is the complement of $I$ in  $M(2g+1)$, then
$$\Eta_{I}(x)\Eta_{\complement I}(x)=\Eta_{M(2g+1)}(x)=
\frac{(x+1)^{2g+1}-x^{2g+1}}{2g+1}.$$

\begin{Rem}
\label{cardg}
Since $\#(M(2g+1))=2g$, the equality $\#(I)=g$ holds if and only if
$\#(\complement I)=g$.
\end{Rem}

\begin{thm}
\label{theorem20}
\begin{itemize}
\item[(i)]
Nice pairs $(u_1(x),u_2(x))$ of degree $g$ over $K$ are exactly the pairs
$\left(\mu \Eta_{I}(x), \frac{2g+1}{\mu}\Eta_{\complement I}(x)\right)$, where $I$ is any $g$-element subset of $M(2g+1)$
and $\mu$ is any element of $K^{*}$.
\item[(ii)]
Let $I$ be a $g$-element subset of $M(2g+1)$.  If $\mu\in K^{*}$, then the corresponding  polynomial
\begin{equation}\begin{aligned}
\label{fI}
f_{I,\mu}(x)&:=f_{0,-1;\mu \Eta_{I}, \frac{2g+1}{\mu}\Eta_{\complement I}}
= x^{2g+1}+\left(\frac{\mu \Eta_{I}(x)+ \frac{2g+1}{\mu}\Eta_{\complement I}(x)}{2}\right)^2\\&
=(x+1)^{2g+1}+\left(\frac{\mu \Eta_{I}(x)- \frac{2g+1}{\mu}\Eta_{\complement I}(x)}{2}\right)^2\end{aligned}\end{equation}
decorated by $\left(\mu \Eta_{I}(x), \frac{2g+1}{\mu}\Eta_{\complement I}(x)\right)$
has no multiple roots for all but finitely many $\mu$.
\item[(iii)]
If $(\mathcal{C}_f,\infty,P,Q)$ is a normalized enhanced genus $g$ hyperelliptic curve $y^2=f(x)$ over $K$, then there is
precisely one pair $(I,\mu)$, where   $I$ is a $g$-element subset of $M(2g+1)$  and  $\mu \in K^{*}$ such that
$f(x)=f_{I,\mu}(x)$ and
\begin{equation}
\label{PsiPQ}
P=\left(0, \frac{\mu \Eta_{I}(0)+\frac{2g+1}{\mu}\Eta_{\complement I}(0)}{2}\right), \
Q=\left(-1, \frac{\mu \Eta_{I}(-1)-\frac{2g+1}{\mu}\Eta_{\complement I}(-1)}{2}\right).
\end{equation}
\item[(iv)]
Let $I$ be a $g$-element subset of $M(2g+1)$ and  $\mu\in K^{*}$ such that $f_{I,\mu}(x)$ has no multiple roots.
Then $\mathcal{C}_{f_{I,\mu}}: y^2=f_{I,\mu}(x)$ is an odd degree genus $g$ hyperelliptic curve over $K$, and
\eqref{PsiPQ} defines torsion points $P,Q \in \mathcal{C}_{f_{I,\mu}}(K)$ of order $2g+1$. In other words,
$(\mathcal{C}_{f_{I,\mu}},\infty,P,Q)$ is a normalized enhanced genus $g$ hyperelliptic curve.
\end{itemize}
\end{thm}

\begin{proof}
(i) Since $\fchar(K)$ does {\sl not} divide $2g+1$, the polynomial
$$(x+1)^{2g+1}-x^{2g+1}$$ has no multiple roots, its degree is $2g$, and the leading coefficient is $2g+1$.
It follows that each factor of $(x+1)^{2g+1}-x^{2g+1}$ is of the form $\mu \Eta_{I}(x)$,
where $I$ is a subset of $M(2g+1)$ and $\mu\in K^{*}$. This implies that
for  every factorization of   $(x+1)^{2g+1}-x^{2g+1}$ into a product of two polynomials
 $u_1(x)$ and $u_2(x)$ we have
 \begin{equation}
 \label{u1u2Psi}
 u_1(x)=\mu \Eta_{I}(x), \ u_2(x)= \frac{2g+1}{\mu}\Eta_{\complement I}(x),
 \end{equation}
 where $I$ is a subset of $M(2g+1)$ and $\mu$ is an element of $K^{*}$.  Nice pairs $(u_1(x),u_2(x))$
 must satisfy $\deg(u_1)=\deg(u_2)=g$.  In light of \eqref{u1u2Psi} and Remark \ref{cardg}, this condition
 is satisfied if and only if  $\#(I)=g$.

 Conversely, if  $I$ is a   $g$-element subset of $M(2g+1)$ and $\mu$ is an  element of $K^{*}$, then
 $$\left(\mu \Eta_{I}(x)\right)\left(\frac{2g+1}{\mu}\Eta_{\complement I}(x)\right)=(x+1)^{2g+1}-x^{2g+1},$$
 $$\deg\left(\mu \Eta_{I}\right)=g=\deg\left(\frac{2g+1}{\mu}\Eta_{\complement I}\right),$$
 i.e., $\left(\mu \Eta_{I}(x), \ \frac{2g+1}{\mu}\Eta_{\complement I}(x)\right)$ is a nice pair.
 This proves (i).

 (ii) follows from Remark \ref{niceR}(iii).

 (iii) follows from (i)  combined with Theorem \ref{restate}.

(iv) follows from (i)  combined with Theorem \ref{restate}.

\end{proof}

 \begin{example}
 Let $g=2$. Then there are exactly 3 families of genus 2 hyperelliptic curves with two pairs of torsion points of order 5.
(See \cite[Sect. 3]{Box2}.)
 \end{example}

 \subsection{ The case when $\fchar(K)$ divides $2g+1$}
 \label{charDiv}
 We write $\mathbb{Z}_{+}$ for the set of nonnegative integers.
 Let us assume that  $\fchar(K)=p>0$ and  $2g+1=p^k(2l+1)$,
 where $k$ and $l$ are positive integers and  $p\nmid (2l+1)$.  We put
   $$M(2l+1):=\{\eps \in K, \eps^{2l+1}=1, \eps\ne 1\}, \ \eta(\eps)=\frac{1}{\eps-1} \ \forall \eps\in M(2l+1).$$
  If  $\upsilon: M(2l+1)\to\mathbb Z_{+}$ is a function on  $M(2l+1)$ with values in $Z_{+}$, then we define its {\sl degree}
  $$\deg(\upsilon)=\sum_{\eps\in M(2l+1)}\upsilon(\eps) \in Z_{+}$$
  and the {\sl monic} polynomial
  \begin{equation}
  \label{degreeUps}
  \Upsilon_{\upsilon}(x)=\prod\limits_{\eps\in M(2l+1)}(x-\eta(\eps))^{\upsilon(\eps)}\in K[x]; \  \deg(\Upsilon_{\upsilon})=\deg(\upsilon).
  \end{equation}
The polynomial
\begin{equation}
 \begin{aligned}\left((x+1)^{2l+1}-x^{2l+1}\right)^{p^k}=
 (x+1)^{2g+1}-x^{2g+1}=(x^{p^k}+1)^{2l+1}-(x^{p^k})^{2l+1} \\=(2l+1)x^{2l p^k}+\binom{2l+1}{2}x^{(2l-1)p^k}+\cdots+\binom{2l+1}{1}x^{p^k}+1\end{aligned}
\end{equation}
has degree
 $2lp^k$ and leading coefficient $2l+1$. Its roots have multiplicity
  $p^k$ and coincide with the roots of the polynomial
   $(x+1)^{2l+1}-x^{2l+1}$.  Hence the set of  roots coincides with
   $$\{\eta(\eps)\mid \eps\in M(2l+1)\}.$$

   We will need the following elementary statement.
   \begin{lem}
   \label{factor Upsilon}
   Let  $\upsilon: M(2l+1)\to\mathbb Z_{+}$ be a function and $\mu \in K^{*}$.

   \begin{itemize}
   \item[(i)]
   The derivative $(\mu \Upsilon_{\upsilon}(x))^{\prime} \ne 0$ if and only if
   there is $\eps\in M(2l+1)$ such that $p$ does not divide $\upsilon(\eps)$.
   \item[(ii)] The polynomial $\mu \Upsilon_{\upsilon}(x)$ divides  $(x+1)^{2g+1}-x^{2g+1}$
   if and only if
   \begin{equation}
   \label{upsilonpk}
   \upsilon(\eps) \le p^k \ \forall \eps\in M(2l+1).
   \end{equation}
   \item[(iii)]  If  inequalities \eqref{upsilonpk} hold, then
   \begin{equation}
   \label{upsbarups}
    (x+1)^{2g+1}-x^{2g+1}=(\mu \Upsilon_{\upsilon}(x))\cdot \frac{2l+1}{2}\Upsilon_{\bar{\upsilon}}(x),
    \end{equation}
    where the function $\bar{\upsilon}: M(2l+1)\to\mathbb Z_{+}$ is defined by
    \begin{equation}
   \label{barupsilon}
    \bar{\upsilon}(\eps)= p^k-\upsilon(\eps )  \ \forall \eps\in M(2l+1).
    \end{equation}
    In addition, $(\mu \Upsilon_{\upsilon}(x))^{\prime} \ne 0$ if and only if
    $\left(\frac{2l+1}{2}\Upsilon_{\bar{\upsilon}}(x)\right)^{\prime}\ne 0$.
    If a polynomial $u(x)\in K[x]$ divides $(x+1)^{2g+1}-x^{2g+1}$, then there exist
    precisely one $\upsilon: M(2l+1)\to\mathbb Z_{+}$ and one $\mu \in K^{*}$ such that
    $u(x)=\mu \Upsilon_{\upsilon}(x)$. In addition,  $\upsilon$ satisfies \eqref{upsilonpk}.
   \end{itemize}
   \end{lem}

\begin{proof}
(i) The derivative  of a nonzero polynomial $u(x)\in K[x]$ is not $0$ if and only if this polynomial is not a $p$th power in $K[x]$
of a polynomial, i.e., if it  has a root  whose multiplicity is not divisible by $p$. Since the set of roots of $\mu \Upsilon_{\upsilon}(x)$
coincides with $\{\eta(\eps)\mid \eps \in M(2l+1), \upsilon(\eps)\ne 0\}$ and the multiplicity of $\eta(\eps)$ equals
$\upsilon(\varepsilon)$, we obtain that there is $\eps \in M(2l+1)$ such that $\upsilon(\varepsilon)\ne 0$ and $p$ does {\sl not} divide  $\upsilon(\varepsilon)$.
This ends the proof of (i).

 (ii) Recall that each $\eta(\eps)$ is a root of $(x+1)^{2g+1}-x^{2g+1}$ with multiplicity $p^k$. This implies
 that  $\mu \Upsilon_{\upsilon}(x)$ divides  $(x+1)^{2g+1}-x^{2g+1}$ if and only if $\eta(\eps)$, viewed as a root
 of $\mu \Upsilon_{\upsilon}(x)$, has multiplicity $\le p^k$, i.e., $\upsilon(\varepsilon)\le p^k$.
 This ends the proof of (ii).

 Assume now that $(\mu \Upsilon_{\upsilon}(x))^{\prime} \ne 0$. By (i),  there is $\eps \in M(2l+1)$ such that $\upsilon(\varepsilon)\ne 0$ and $p$ does {\sl not} divide  $\upsilon(\varepsilon)$.
 Then $  \bar{\upsilon}(\eps)= p^k-\upsilon(\eps )$ is also {\sl not} divisible by $p$.
 Assertions (iii) and (iv) are obvious.

\end{proof}

\begin{defn}
We call a function $\upsilon: M(2l+1)\to\mathbb Z$ {\sl admissible}
if it enjoys the following properties.

\begin{itemize}
\item[(i)]
$$0\leq \upsilon(\eps )\leq p^k \ \forall \eps\in M(2l+1).$$
\item[(ii)]
There exists $\eps\in M(2l+1)$ such that $\upsilon(\eps )\not\equiv0\pmod p$.
\item[(iii)]
$$\sum\limits_{\eps\in M(2l+1)}\upsilon(\eps)\leq g,   \ \sum\limits_{\eps\in M(2l+1)}(p^k-\upsilon(\eps))\leq g.$$
\end{itemize}
\end{defn}

\begin{Rem}
If  $\upsilon: M(2l+1)\to\mathbb Z$ is an admissible function, then
$$\bar{\upsilon}: M(2l+1)\to\mathbb Z, \ \eps\mapsto p^k-\upsilon(\eps )$$
is also an admissible function.
\end{Rem}

\begin{example}
\label{exAd}
Let us partition the $2l$-element set $ M(2l+1)$ into a disjoint union of two $l$-element sets $I$ and $J$ and define a function
$$\upsilon_{I,J}:M(2l+1)\to\mathbb Z_{+},$$
$$ \upsilon_{I,J}(\eps)=\frac{p^k+1}{2}<p^k \ \forall \eps \in I, \ 
\upsilon_{I,J}(\eps)=\frac{p^k-1}{2}<p^k \ \forall \eps \in J.$$
None of $\upsilon_{I,J}(\eps)$ is divisible by $p$ and
$$\upsilon_{J,I}(\eps)=p^k-\upsilon_{I,J}(\eps) \ \forall \eps\in M(2l+1).$$
The function $\upsilon_{I,J}$ is
 {\sl admissible}, because
 $$\sum_{\eps \in I}\frac{p^k+1}{2}+\sum_{\eps \in J}\frac{p^k-1}{2}=l p^k <\frac{(2l+1)p^k-1}{2}=g.$$

\end{example}

\begin{thm}
\label{theorem22}
\begin{itemize}
\item[(i)]
Nice pairs $(u_1(x),u_2(x))$ of degree $g$ over $K$ are exactly the pairs
$\left(\mu \Upsilon_{\upsilon}(x),  \frac{2l+1}{\mu}\Upsilon_{\bar \upsilon}(x)\right)$, where $\upsilon$ is an  admissible
function on $M(2l+1)$ with
$$\deg(\upsilon)\le g, \ \deg(\bar{\upsilon})\le g$$
and $\mu\in K^{*}$.
\item[(ii)]
Let $\upsilon$ be an  admissible
function on $M(2l+1)$.  If $\mu\in K^{*}$, then the corresponding  polynomial
\begin{equation}
\label{fupsilon}\begin{aligned}
f_{\upsilon,\mu}(x)&:=f_{0,-1;\mu  \Upsilon_{\upsilon}, \frac{2l+1}{\mu}\Upsilon_{\bar \upsilon}}=
 x^{2g+1}+\left(\frac{\mu  \Upsilon_{\upsilon}(x)+ \frac{2l+1}{\mu}\Upsilon_{\bar \upsilon}(x)}{2}\right)^2\\&=
(x+1)^{2g+1}+\left(\frac{\mu \Upsilon_{\upsilon}(x)- \frac{2l+1}{\mu}\Upsilon_{\bar \upsilon}(x)}{2}\right)^2\end{aligned}\end{equation}
decorated by $\left(\mu \Upsilon_{\upsilon}(x), \frac{2l+1}{\mu}\Upsilon_{\bar \upsilon}(x)\right)$
has no multiple roots for all but finitely many $\mu$.
\item[(iii)]
If $(\mathcal{C}_f,\infty,P,Q)$ is a normalized enhanced genus $g$ hyperelliptic curve $y^2=f(x)$ over $K$, then there is
precisely one pair $(\upsilon, \mu)$, where $\upsilon$ is an admissible function  on $M(2l+1)$ and $\mu \in K^{*}$ such that
$f(x)=f_{\upsilon,\mu}(x)$ and
\begin{equation}
\label{UpsPQ}
P=\left(0, \frac{\mu \Upsilon_{\upsilon}(0)+\frac{2l+1}{\mu}\Upsilon_{\bar \upsilon}(0)}{2}\right), \
Q=\left(-1, \frac{\mu \Upsilon_{\upsilon}(-1)-\frac{2l+1}{\mu}\Upsilon_{\bar \upsilon}(-1)}{2}\right).
\end{equation}
\item[(iv)]
Let $\upsilon$  be an admissible function on $M(2l+1)$ and  $\mu\in K^{*}$ such that $f_{\upsilon,\mu}(x)$ has no multiple roots.
Then $\mathcal{C}_{f_{\upsilon,\mu}}: y^2=f_{\upsilon,\mu}(x)$ is an odd degree genus $g$ hyperelliptic curve over $K$, and
\eqref{UpsPQ} defines torsion points $P,Q \in \mathcal{C}_{f_{\upsilon,\mu}}(K)$ of order $2g+1$.
\end{itemize}
\end{thm}

\begin{proof}
(i) follows from Lemma \ref{factor Upsilon} and \eqref{degreeUps}.

(ii) follows from (i) combined with Remark \ref{niceR}(iii).

(iii) follows from (i)  combined with Theorem \ref{restate}.

(iv) follows from (i)  combined with Theorem \ref{restate}.
\end{proof}

 \section{Computations of Weil pairings}
 \label{WeilComp}
 We will use the notation of Subsection \ref{charNotDiv}.
 In this section we assume that $\fchar(K)$ does {\sl not} divide $2g+1$; our goal is to compute the value of the Weil pairing
 between torsion points
  $P$ and  $Q$ in $\mathcal{C}(K)$ of order $2g+1$, where  $\alb (P)\neq \pm \alb (Q)$. We may assume that the curve is defined by the equation
$y^2=x^{2g+1}+v_1(x)^2,$ where
$$v_1(x)=\frac{\mu}2\Eta_I(x)+\frac{2g+1}{2\mu}\Eta_{\complement I}(x), $$
while
$$x^{2g+1}+v_1^2=(x+1)^{2g+1}+v_2(x)^2,$$
where
$$v_2(x)=\frac{\mu}2\Eta_I(x)-\frac{2g+1}{2\mu}\Eta_{\complement I}(x).$$
In this case we may assume that
$P=(0,v_1(0))$ and $Q=(-1,v_2(-1)).$

Let us consider the degree zero divisors $D_P=(P)-(\infty)$ and  $D_Q=(Q)-(\infty)$ on $\mathcal{C}$. We know that their linear equivalence classes
 have order $2g+1$.
Let us consider a Weierstrass point $\W=(w,0)$ on our curve, where $\al$ is a root of  $x^{2g+1}+v_1(x)^2$.  The linear equivalence class
of the divisor  $D_{\W}:=(\W)-(\infty)$ has order  $2$. Therefore, the linear equivalence class
of the divisor
 $$D=D_P-D_{\W}=(P)-(\W)$$ has order $2(2g+1)$.
Since $\div(x-w)=2(\W)-(\infty))$, the divisor  $2D$ is linearly equivalent to $2D_P$.

We have
$$\begin{aligned}e_{2(2g+1)}(P,Q)=e_{2(2g+1)}(D,D_Q)=e_{2g+1}(2D,D_Q)\\=e_{2g+1}(2D_P,D_Q)
=(e_{2g+1}(D_P,D_Q))^2.
\end{aligned}$$
Let us put
$$g_Q=(y-v_2(x))^2.$$
Then
$$\div(g_Q)=2\div(y-v_2(x))=2(2g+1)(Q)-2(2g+1)(\infty).$$
Let us put
$$g_P=\frac{(y-v_1(x))^2}{(x-w)^{2g+1}}.$$
Since
$$\div(y-v_1(x))=(2g+1)(P)-(2g+1)(\infty)$$
and
$$\div(x-w)=2(\W)-2(\infty),$$
we have
$$\div(g_P)=2(2g+1)(P)-2(2g+1)(\W).$$
Evaluating
$g_P(D_Q)$, we get
$$\begin{aligned}g_P(D_Q)=\frac {g_P(Q)}{g_P(\infty)}=-\frac{(v_2(-1)-v_1(-1))^2}{(1+w)^{2g+1}}\\=-
\left(\frac{2g+1}{\mu}\right)^2\frac{\Eta_{\complement I}^2(-1)}{(1+w)^{2g+1}},
\end{aligned}$$
since $g_P(\infty)=1$.
Now let us evaluate $g_Q(D)$.
We have
$$g_Q(D)=\frac{g_Q(P)}{g_Q(W)}=\frac{((v_1(0)-v_2(0))^2}{v_2(w)^2}
=\left(\frac{2g+1}{\mu}\right)^2\frac{\Eta_{\complement I}^2(0)}{v_2(w)^2}.$$
Notice that since
 $w$ is a root of $(x+1)^{2g+1}+v_2^2(x),$ then
$$v_2(w)^2=-(1+w)^{2g+1},$$
which gives us
$$g_Q(D)=-\left(\frac{2g+1}{\mu}\right)^2\frac{\Eta_{\complement I}^2(0)}{(1+w)^{2g+1} }.$$
Therefore,
$$e_{2(2g+1)}(P,Q)=\frac{g_P(D_Q)}{g_Q(D)}= \frac{\Eta_{\complement I}^2(-1)}
{\Eta_{\complement I}^2(0)}=\frac{\prod\limits_{i\in \complement I}(1+\eta(\eps))^2}{\prod\limits_{i\in \complement I} \eta(\eps)
^2}=\left(\prod\limits_{\eps\in \complement I}\eps \right)^2, $$
since
$(1+\eta(\eps))/ \eta(\eps) =\eps.$
This implies that
$$e_{2g+1}(P,Q)=\pm\prod\limits_{\eps\in \complement I}\eps.$$
Since $e_{2g+1}(P,Q)$ and all $\eps$ are $(2g+1)$th  roots of unity, and  $2g+1$ is {\sl odd}, we get at last
$$e_{2g+1}(P,Q)=\prod\limits_{\eps\in \complement I}\eps.$$

\end{document}